\documentclass[sort&compress, ACS,STIX1COL, leqno, 12pt, a4]{WileyNJD-v2}

\usepackage{graphicx}

\articletype{}%

\received{\today}
\revised{~}
\accepted{~}

\usepackage{amsmath,amssymb,color,comment}
\usepackage{mathtools}
\mathtoolsset{showonlyrefs=true}
\usepackage{cancel}
\newcommand{\bfu}{\boldsymbol{u}}
\newcommand{\bff}{\boldsymbol{f}}
\newcommand{\bfd}{\boldsymbol{d}}
\newcommand{\bfv}{\boldsymbol{v}}
\newcommand{\BonA}{\frac{B}{A}}
\newcommand{\afunc}{a}
\newcommand{\bfunc}{b}
\newcommand{\gfunc}{g}
\newcommand{\hfunc}{h}

\newcommand{\gcoef}{\mathfrak{g}}

\newcommand{\rorel}{\sigma}
\newcommand{\romean}{\rho_0}
\newcommand{\cmean}{c_0}

\newcommand{\Xf}{{X_{\bff}}}

\newcommand{\Xu}{{X_{\bfu}}}
\newcommand{\Xrho}{{X_{\rorel}}}

\newcommand{\XItp}{X_{\It p}}

\newcommand{\ma}{\text{(ma)}\ }
\newcommand{\pd}{\text{(pd)}\ }

\newcommand{\dS}{\, \textup{d}S}
\newcommand{\ds}{\, \textup{d}s}
\newcommand{\dx}{\, \textup{d}x}

\newcommand{\Hone}{H^1(\Omega)}

\newcommand{\Ltwo}{L^2(\Omega)}
\newcommand{\Ltwod}{L^2(\Omega)^d}

\newcommand{\Linf}{L^\infty(\Omega)}

\newcommand{\ddt}{\frac{\textup{d}}{\textup{d}t}}

\newcommand{\intO}{\int_\Omega}

\newcommand{\bfuzero}{\boldsymbol{u}_0}
\newcommand{\rorelzero}{\rorel_0}

\definecolor{grey}{rgb}{0.5,0.5,0.5}
\definecolor{brown}{rgb}{0.7,0.7,0.}

\def\It{\textup{I}_t}

\def\wtildeL{\widetilde{L}}

\def\pOmega{\partial \Omega}

\def\Xsigma{X_\sigma}

\def\Wn{W^n}

\def\N{\mathbb{N}}

\def\sigman{\sigma^n}
\def\sigmant{\sigma^{n}_t}
\def\rorelnt{\sigmant}

\def\pn{p^n}

\def\bfun{\bfu^n}
\def\bfunt{\bfu^{n}_t}

\def\xisigma{\xi^\sigma}
\def\xip{\xi^p}

\def\R{\mathbb{R}}

\def\roreln{\rorel^n}

\def\bfunzero{\bfu^n_0}

\def\wi{w_i}

\def\romeaninv{\frac{1}{\romean}}

\definecolor{darkgreen}{rgb}{0,0.5,0}

\def\Xumu{X^\mu_{\bfu}}
\def\solspacereg{\mathcal{X}^\mu}

\def\Hdiv{H(\text{div};\Omega)}

\def\LinftLinf{L_t^\infty(\Linf)}
\def\LinfLinf{L^\infty(\Linf)}

\def\LtwoTHone{L^2(0,T; \Hone)}

\def\LtwoLtwo{L^2(L^2(\Omega))}
\def\LtwoTLtwo{L^2(0,T; L^2(\Omega))}

\def\LoneLtwo{L^1(\Ltwo)}

\def\LinftLtwo{L^\infty_t(\Ltwo)}

\def\LtwotHone{L^2_t(\Hone)}

\def\LinfTLtwo{L^\infty(0,T; \Ltwo)}
\def\LoneLtwo{L^1(\Ltwo)}

\def\LtwotLtwo{L^2_t(\Ltwo)}

\def\Lsix{L^6(\Omega)}

\def\Wonethree{W^{1,3}(\Omega)}

\def\Lthree{L^3(\Omega)}

\def\LinfLtwo{L^\infty(\Ltwo)}
\def\Lsix{L^6(\Omega)}

\def\LtwotLsix{L^2_t(\Lsix)}
\def\LinftLthree{L^\infty_t(\Lthree)}
\def\LtwotLthree{L^2_t(\Lthree)}

\def\HoneTLtwod{H^1(0,T; \Ltwod)}
\def\LinfHdiv{L^\infty(\Hdiv)}
\def\LinfTHdiv{L^\infty(0,T; \Hdiv)}
\def\LtwoTLtwod{L^2(0,T; \Ltwod)}
\def\Hytwo{H^{\frac{y}{2}}(\Omega)}
\def\Hyplusonetwo{H^{\frac{y+1}{2}}(\Omega)}
\def\HoneLtwo{H^1(\Ltwo)}
\def\LinfnLtwo{L^\infty(0,\Tn; \Ltwo)}
\def\LinfnLthree{L^\infty(0,\Tn; \Lthree)}
\def\LtwonLsix{L^2(0,\Tn; \Lsix)}

\def\LtwonLthree{L^2(0,\Tn; \Lthree)}

\def\Lfour{L^4(\Omega)}

\def\LtwotLinf{L^2_t(\Linf)}
\def\LtwotHyplusonetwo{L^2_t(\Hyplusonetwo)}
\def\LtwotHytwo{L^2_t(\Hytwo)}
\def\LinftHyplusonetwo{L^\infty_t(\Hyplusonetwo)}
\def\LonetLtwo{L^1_t(\Ltwo)}

\def\LinfHyplusonetwo{L^\infty(\Hyplusonetwo)}

\def\Htwo{H^2(\Omega)}

\def\Hyplusonehalf{H^{\frac{y+1}{2}}(\Omega)}
\def\Hyhalf{H^{\frac{y}{2}}(\Omega)}
\def\LinfHyplusonehalf{L^\infty(\Hyplusonehalf)}
\def\LtwoHyhalf{L^2(\Hyhalf)}
\def\originalL{\wtildeL}
\def\modL{L}

\def\rorelt{\rorel_t}

\def\rhs{\textup{rhs}}

\def\divbfu{\nabla \cdot \bfu}

\def\ulacoef{\underline{\afunc}}

\def\olacoef{\overline{\afunc}}

\def\ulafunc{\underline{\afunc}}
\def\ulbfunc{\underline{\bfunc}}
\def\olafunc{\overline{\afunc}}
\def\olbfunc{\overline{\bfunc}}

\def\rorelnzero{\roreln_{0}}
\def\bfunzero{\bfun_{0}}

\def\xin{\xi^n}

\def\momu{(\text{mo}^\mu) }
\def\moG{(\text{mo}^{\text{G}}) }
\def\maG{(\text{ma}^{\text{G}}) }
\def\pdG{(\text{pd}^{\text{G}}) }

\def\Deltaromean{\Delta_{1/\romean}}

\def\xinsigma{\xi^{\sigma, n}}
\def\xinp{\xi^{p, n}}
\def\divbfun{\nabla \cdot \bfu^n}

\def\projection{\textup{P}^{\romean}_{\Wn}}
\def\gprojection{\textup{P}^{\romean}_{\Wn}\gfunc}
\def\gcoefprojection{\textup{P}^{\romean}_{\Wn}\gcoef}

\def\intTO{\int_0^T \int_{\Omega}}
\def\dxt{\, \textup{d}x\textup{d}t}

\def\calT{\mathcal{T}}

\def\pnstar{\pn_*}
\def\rorelnstar{\roreln_*}

\def\ball{B}

\def\rorelnone{\rorel^{n,(1)}}
\def\rorelntwo{\rorel^{n,(2)}}
\def\rorelnstarone{\rorel_*^{n,(1)}}
\def\rorelnstartwo{\rorel_*^{n,(2)}}

\def\pnone{p^{n,(1)}}
\def\pntwo{p^{n,(2)}}
\def\pnstarone{p_*^{n,(1)}}
\def\pnstartwo{p_*^{n,(2)}}

\DeclareMathOperator*{\esssup}{ess\,sup}
\def\Tn{T_n}
\def\bfvn{\bfv^n}
\def\calL{\mathcal{L}}
\def\rhsone{\rhs_1}
\def\rhstwo{\rhs_2}

\def\intt{\int_0^t}

\def\calE{\mathcal{E}}
\def\calD{\mathcal{D}}

\def\eps{\varepsilon}
\def\Xromean{X_{\romean}}
\def\XBonA{X_{B/A}}
\def\Xcmean{X_{\cmean}}

\def\Is{\textup{I}_s}

\def\Gronwall{Gr\"onwall}

\def\deltaromeancmean{\delta_{\romean, \cmean}}

\def\oneoverfoureps{\frac{1}{4\eps}}
\def\oneovertwoeps{\frac{1}{2\eps}}

\def\intT{\int_0^T}
\def\dt{\, \textup{d}t}

\def\bfumuzero{\bfu^{\mu=0}}
\def\divbfumuzero{\nabla \cdot \bfumuzero}
\def\pmuzero{p^{\mu=0}}

\def\rorelmuzero{\rorel^{\mu=0}}

\def\calX{\mathcal{X}}

\def\ororelnstar{\overline{\rorel}_*^n}
\def\ororeln{\overline{\rorel}^n}

\def\opnstar{\overline{p}_*^n}
\def\opn{\overline{p}^n}

\def\bfdzero{\bfd_0}

\begin{document}

\title{Existence of solutions to k-Wave models of nonlinear ultrasound propagation in biological tissue}

\author[1]{Ben Cox}

\author[2]{Barbara Kaltenbacher}

\author[3]{Vanja Nikoli\'c}

\author[4]{Felix Lucka}

\authormark{}

\address[1]{\orgdiv{Department of Medical Physics and Biomedical Engineering}, \orgname{University College
		London}, \orgaddress{London, WC1E 6BT, United Kingdom}}

\address[2]{\orgdiv{Department of Mathematics}, \orgname{Alpen-Adria-Universit\"at Klagenfurt}, \orgaddress{Universit\"atsstra\ss e 65--67, A-9020 Klagenfurt, Austria}}

\address[3]{\orgdiv{Department of Mathematics}, \orgname{Radboud University}, \orgaddress{Heyendaalseweg 135, 6525 AJ Nijmegen, The Netherlands}}

\address[4]{\orgdiv{Centrum Wiskunde \& Informatica}, \orgaddress{Science Park 123, Amsterdam, The Netherlands}}

\corres{*Vanja Nikoli\'c, Radboud University, Heyendaalseweg 135, 6525 AJ Nijmegen, The Netherlands. \email{vanja.nikolic@ru.nl}}


\abstract[Summary]
{We investigate models for nonlinear ultrasound propagation in soft biological tissue based on the one that serves as the core for the software package k-Wave. The systems are solved for the acoustic particle velocity, mass density, and acoustic pressure and involve a fractional absorption operator. We first consider a system that incorporates additional viscosity in the equation for momentum conservation. By constructing a Galerkin approximation procedure, we prove the local existence of its solutions. In view of inverse problems arising from imaging tasks, the theory allows for the variable background mass density, speed of sound, and the nonlinearity parameter in the systems. Secondly, under stronger conditions on the data, we take the vanishing viscosity limit of the problem, thereby rigorously establishing the existence of solutions for the limiting system as well. 

}
\keywords{ultrasound modeling, k-Wave, local existence, fractional Laplacian}

\maketitle


\section{Introduction}

Ultrasound waves propagating in soft biological tissue, even at the intensities used in biomedical imaging applications, can undergo noticeable nonlinear distortion. At higher intensities still, such as are used in therapeutic medical applications, the effect of the nonlinearities can be very significant. Several scientific software packages have therefore been developed for modelling nonlinear propagation in biological tissue~\cite{list_of_nonlinear_solvers}. Here, the system of equations that are the basis for one of those packages, k-Wave~\cite{TreebyJarosRendellCox2012ModelingNU,kwave}, will be analysed. It is given in terms of the acoustic particle velocity $\bfu$, mass density $\rho$, and acoustic pressure $p$ by the following set of equations:
\begin{equation} \label{original system}
	\begin{aligned}
		&\text{linear momentum conservation: }&& \romean \bfu_t+\nabla p =\bff,\\
		&\text{mass conservation: }	 && \rho_t + (2\rho+\romean )\, \nabla\cdot\bfu +\bfu\cdot\nabla\romean = 0,\\
		&\text{pressure-density relation: } && p- \cmean^2\Bigl(\rho +\bfd\cdot\nabla\romean + \frac{B}{2A} \frac{\rho^2}{\romean}
		-\originalL\rho\Bigr)=0,
	\end{aligned}
\end{equation}
where $\bfu=\bfd_t$; see \cite[system (10)]{TreebyJarosRendellCox2012ModelingNU} and \cite[system (1)]{JarosRendellTreeby2016}. The operator $\originalL$ accounts for absorption and dispersion. It is defined by
\begin{equation}\label{L_JRT16}
	\originalL\rho= 2\alpha_0 \left(-\cmean^{y-1}(-\Delta)^{\frac{y}{2}-1} \rho_t 
	+\cmean^y \tan \left(\frac{\pi y}{2}\right) (-\Delta)^{\frac{y+1}{2}-1} \rho \right)
\end{equation}
with $y\in(1,3)$ and $\alpha_0>0$; see \cite[eq. (3)]{JarosRendellTreeby2016}. In human tissue typically $y\in(1,2]$. The quantities $\romean$,  $\cmean$, and $\BonA$ in this system are the background mass density, isentropic sound speed, and  nonlinearity parameter, respectively. \\
\indent In k-Wave these equations are discretised using a pseudo-spectral time domain (PSTD) time-stepping scheme with a dispersion correcting factor applied in the spatial Fourier domain. The particular form of the absorption/dispersion term in \eqref{original system} was chosen both because the resulting absorption depends on frequency according to a power law, as empirically observed in many tissue types, and because it is memory-efficient when implemented using a PSTD scheme.

\subsection{Numerical example}

In the spirit of motivation for the study of the system \eqref{original system}, a simple numerical example, computed using k-Wave, will be given here. With ultrasound tomography in mind, this example shows that for a fixed number of sources and detectors more independent data can be obtained when nonlinear effects are included than in the linear case. Specifically, inspecting the singular value spectrum of a set of simulated measurements, shows that when pairs of sources are used simultaneously in the nonlinear regime, the resulting measured signals are not just linear combinations of the signals measured with the individual sources alone, as they are in the linear case. Fig.\ \ref{figure1}(left) shows a ring array of 8 equally-spaced transducer elements that all act as detectors, and 4 of which (shown in white) also act as sources, surrounding a region with a heterogeneous sound speed\cite{lou2017generation}.
All other material properties were chosen to be homogeneous: mass density $\rho_0 = 1000$ kg/m$^3$, absorption coefficient $\alpha = \alpha_0f^y$ where $\alpha_0 = 0.5$ dB/cm/MHz$^y$, $y = 1.5$, $f = 0.25$ MHz is the frequency, $B/A = 7$ is the acoustic nonlinearity parameter, and the source acoustic pressure is $5$ MPa. Simulations were conducted in both the nonlinear and linear regimes (ie. no nonlinear terms included in the equations, equivalent to using a low source amplitude source). For each simulation, the transducers acting as sources were driven with a single-frequency sinusoidal wave, and acoustic pressure time series were detected at all other transducers. When acting as a source, a transducer does not also act as a detector, so in both the linear and nonlinear cases, 64 time series were measured: 28 using single sources (4 sources x 7 detectors), and 36 using pairs of sources driven simultaneously (6 pairs of sources x 6 detectors). Fig.\ \ref{figure1}(right) shows a snapshot of the acoustic pressure field emitted from the leftmost transducer. Fig.\ \ref{figure2}(left) shows examples of measured time series in both the linear and nonlinear cases, showing characteristic wave steepening due to the nonlinearity increasing the wave speed at the peaks of the wave and decreasing it at the troughs. All 64 time series measured in the linear case were stacked into a matrix and the singular values of that data matrix were computed. This was also done in the nonlinear case. The singular value spectra, normalised to the largest singular value, are plotted in Fig.\ \ref{figure2}(right). The cliff-edge after the 28th singular value in the linear case indicates that the data obtained using pairs of sources are merely linear combinations of the data obtained using single sources. This is not the case in the nonlinear regime. While this example may be interesting, we note that it does not prove - or indicate the extent to which - the data carries additional information about the material properties, the estimation of which is the ultimate goal of ultrasound tomography. 

\begin{figure}[h]
	\includegraphics[width=0.5\textwidth]{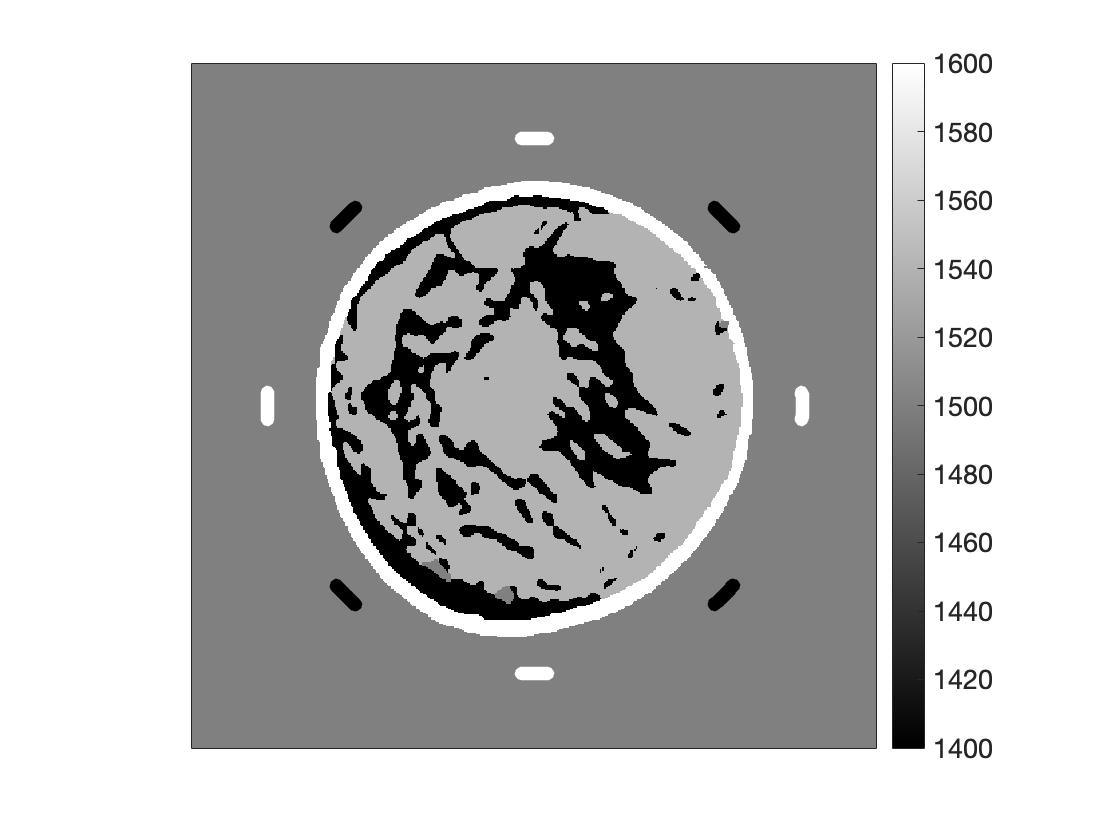}
	\includegraphics[width=0.5\textwidth]{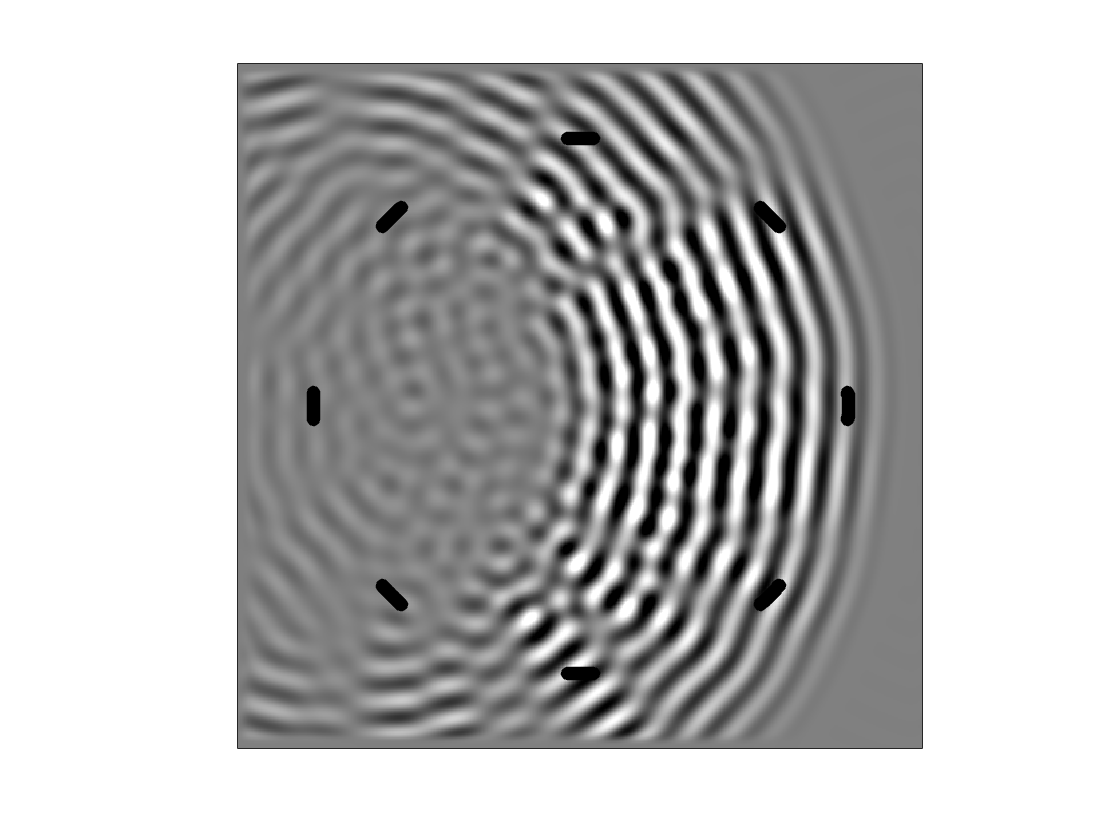}
	\caption{\label{figure1} Left: Set-up of the numerical example, showing the sound speed map (m/s) and the positions of the transducers (white: sources and detectors, black: detectors only). Right: Snapshot of the field from the leftmost transducer acting as a source.}
\end{figure}

\begin{figure}[h]
	\includegraphics[width=0.5\textwidth]{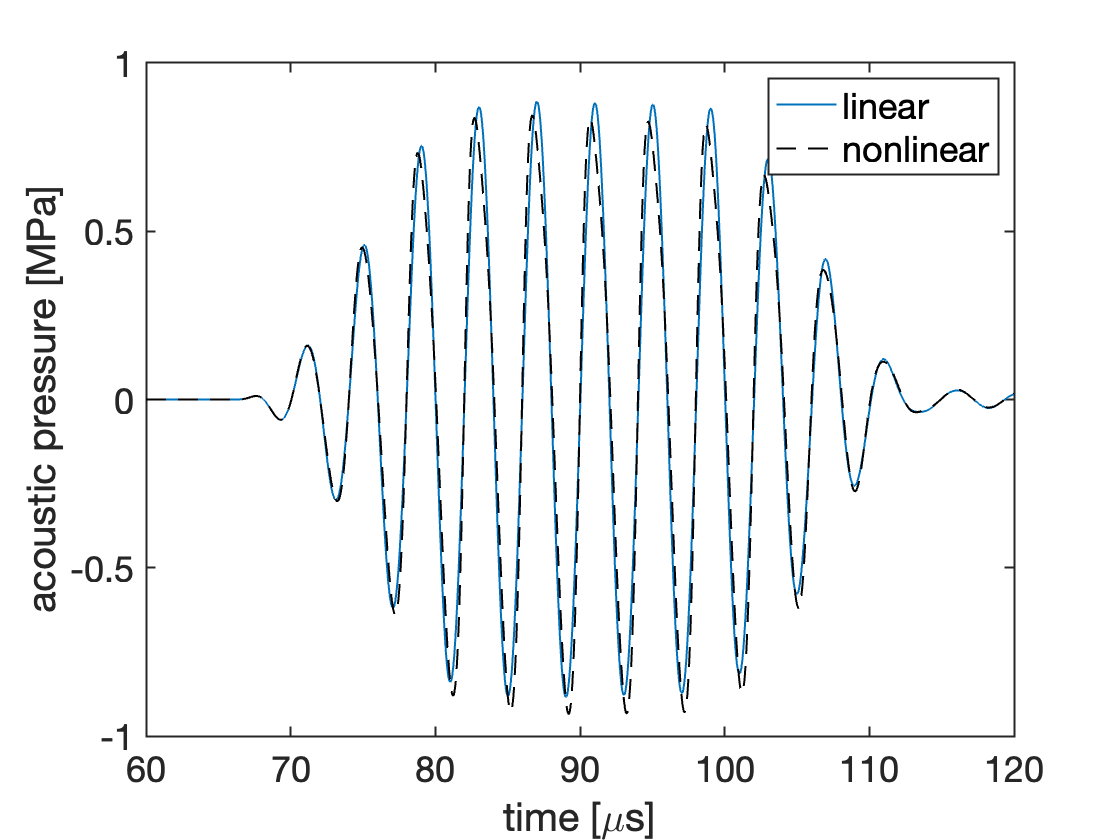}
	\includegraphics[width=0.5\textwidth]{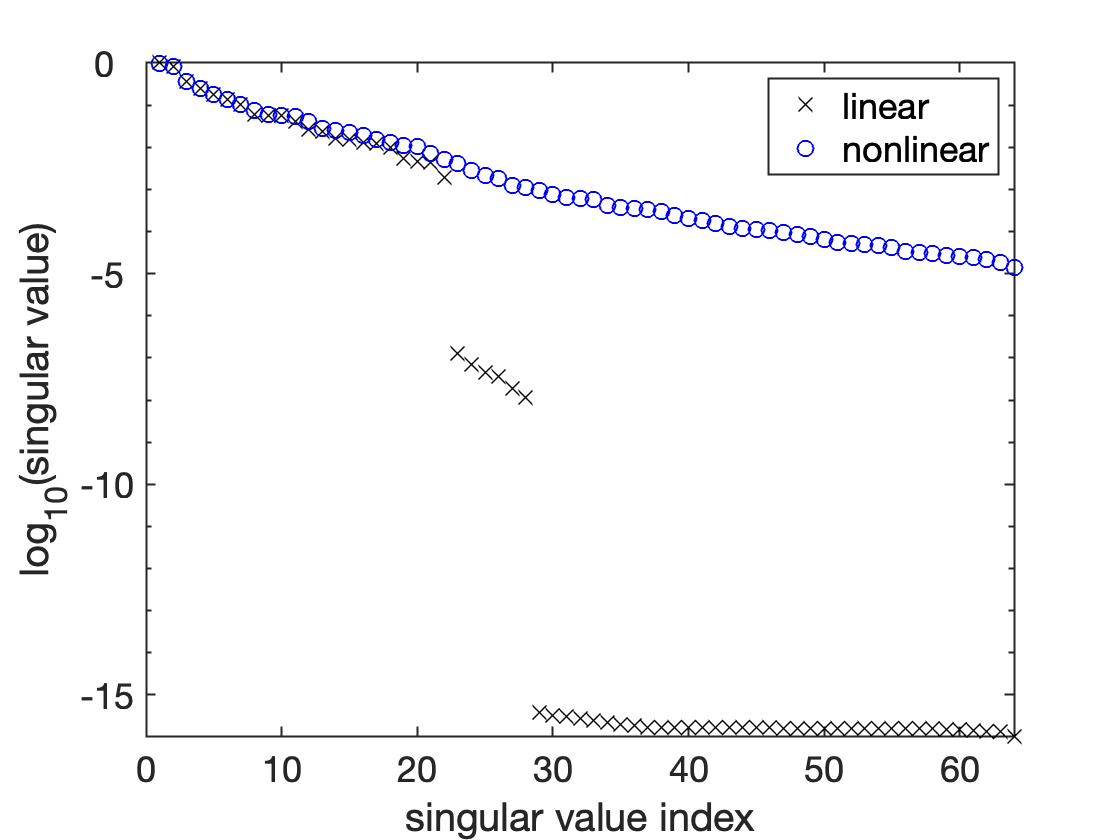}
	\caption{\label{figure2} Left: Examples of linear and nonlinear time series. Right: Singular value spectrum of the linear and nonlinear data.}
\end{figure}

\newpage
\subsection{Main contributions}

The main aim of this work is to gain rigorous understanding of the systems of the form in \eqref{original system} with possible additional viscosity included in the momentum balance equation.  In view of inverse problems arising from imaging tasks, we are particularly interested in allowing $\BonA$, $\cmean$, and $\romean$ in \eqref{original system} to depend on $x$ in this order of importance, that is, the simplification $\romean\equiv const.$ is the least restrictive one. With this in mind, we can rewrite the mass conservation in terms of
$\rorel=\dfrac{\rho}{\romean}$
as follows:
\[
\rorelt +(1+2\rorel) \nabla \cdot \bfu+ \bfu \cdot \nabla \ln \romean=0.
\]
In the analysis, we supplement the system with the following boundary conditions:
\begin{equation} \label{balance-state bc}
	\nu\cdot\bfu=0, \qquad 
	\nu\cdot\nabla \rorel=0 \
	\text{ on }\pOmega,
\end{equation}
where $\nu$ is the outer unit normal vector at the boundary $\pOmega$, as well as the initial velocity and density data
\begin{equation} \label{balance-state ic}
	\bfu(0)=\bfuzero, \qquad 	\bfd(0)=\bfdzero, \qquad
	\rorel(0)=\rorelzero.
\end{equation}
Then $\bfd = \It \bfu + \bfdzero$, where $\It \bfu=\int_0^t \bfu(s)\ds$ for $t \in [0,T]$. By taking into account a viscosity term in the momentum balance in \eqref{original system}
and rearranging the terms, we arrive at the following system for $(\bfu, \sigma, p)$: 
\begin{equation}\label{regularized balance-state}
	\left \{ \ \begin{aligned}                                                                                                                                                                                                                                                                                                                                                                                                                                                                                                                                                                                                                                                                                                                        
		\momu \qquad &\romean \bfu_t+\nabla p -\mu\nabla(\nabla\cdot\bfu) 
		=\bff,\\
		\ma \qquad &\rorelt + \afunc(\rorel)\nabla \cdot \bfu  =-\bfu \cdot \nabla \ln \romean: = g(\bfu) ,
		\\
		\pd \qquad &p - \cmean^2\romean\bfunc(\rorel)\,\rorel + \modL \rorel=\cmean^2\bfd\cdot\nabla\romean=\cmean^2(\It \bfu+\bfdzero)\cdot\nabla\romean:=\hfunc(\bfu),
	\end{aligned} \right.
\end{equation}
with a \emph{modified} absorption operator
\begin{equation}\label{L_almostoriginal}
	\modL\rorel=-2\alpha_0 
	(-\Deltaromean)^{-1}
	\left[\tau(-\Delta)^{\frac{y}{2}} \rorel_t 
	+\eta (-\Delta)^{\frac{y+1}{2}} \rorel\right], \quad \tau, \eta>0,
\end{equation}
\footnote{Note that our analysis could also handle 
	the choice 
	\eqref{L_JRT16},
	however at the cost of involving higher order commutators of the coefficients $\romean$ and $\cmean$ and thus having to impose higher smoothness on them.}
where	 $-\Delta=-\Delta_N$ denotes the homogeneous Neumann-Laplacian and
 \begin{align}
	\Deltaromean v:=\nabla\cdot(\frac{1}{\romean}\nabla v).
\end{align}
We have introduced the following short-hand notation in \eqref{regularized balance-state}: 
\begin{equation}\label{def a b}
	\begin{aligned}
		\afunc(\rorel)=1+2\rorel, \quad
		\bfunc(\rorel)=1+\frac{B}{2A} \rorel
	\end{aligned}
\end{equation}
and
\begin{equation}\label{def gfunc hfunc}
	\begin{aligned}
		g(\bfu) =-\bfu \cdot \nabla \ln \romean, \qquad \hfunc(\bfu) =\cmean^2\bfd\cdot\nabla\romean=\cmean^2(\It \bfu+\bfdzero)\cdot\nabla\romean.
	\end{aligned}
\end{equation}
We have chosen $-\mu\nabla(\nabla\cdot\bfu)$ for the viscosity term with $\mu > 0$, (which is equal to $- \mu\Delta \bfu$ for irrotational $\bfu$) since it allows us to make use of cancellations below without having to impose the equation $\nabla\times\bfu=0$ as a further PDE. \\
\indent The main contributions of the remaining of the work pertain to the analysis of the system in \eqref{regularized balance-state}; in particular, we establish existence of its solutions in Theorem~\ref{theorem wellposedness regularized problem} using a Galerkin-based framework. Additionally, under the assumption that $\gfunc=\hfunc=0$, we conduct analysis in the vanishing viscosity limit $\mu \searrow 0$ as a way of relating system \eqref{regularized balance-state} to system \eqref{original system} with the absorption operator \eqref{L_almostoriginal}. This result is contained in Theorem~\ref{theorem wellposedness  mu ero} below. \\
\indent To the best of our knowledge, systems of the form in \eqref{regularized balance-state} with fractional absorption have not been studied so far in a rigorous manner. In contrast, rigorous techniques for single-equations models in nonlinear acoustics, such as the Westervelt or Kuznetsov equation, are by now pretty well-established; see, for example,~\cite{kaltenbacher2009global, acosta2022nonlinear, mizohata1993global, kaltenbacher2022inverse, Dekkers2020} and the review paper~\cite{kaltenbacher2015mathematics}. Analysis of a local compressible Navier--Stokes system governing nonlinear sound motion can be found in~\cite{tani2017mathematical}; see also~\cite{matsumura1980initial} and the references contained therein.  
\subsection*{Notation} Below we occasionally use $x \lesssim y$ for $x\leq C y$, where $C>0$ is a generic constant that does not depend on the Galerkin discretization parameter. We use subscript $t$ to denote the temporal domain $(0,t)$ in Bochner spaces, where $t$ is taken from a certain time interval to be specified; for example, $\|\cdot\|_{L^p_t(L^q(\Omega))}$ denotes the norm on $L^p(0,t; L^q(\Omega))$. If the subscript is omitted, the temporal domain is meant to be $(0,T)$.

\section{Existence of solutions} \label{Sec:Analysis}

In this section, we provide the proof of existence of solutions of \eqref{regularized balance-state} with boundary and initial data given in \eqref{balance-state bc} and \eqref{balance-state ic}, respectively. We first set the notion of the solution, where equations $\momu$ \ and $\pd$ will be understood in a time-integrated sense. More precisely, the solution space for the velocity is
\begin{equation} \label{def Xumu}
	\Xumu = \left\{\bfu\in \LinfTHdiv: \ \sqrt{\mu}\|\nabla(\nabla\cdot\bfu)\|_{\LtwoLtwo}<\infty, \quad \bfu \cdot \nu=0 \text{ on }\partial \Omega \right\},
\end{equation}
endowed with the norm
\begin{equation}
	\|\bfu\|_{\Xumu} = \left\{ \|\divbfu\|_{\LinfLtwo}^2+ \mu \|\nabla(\nabla\cdot\bfu)\|^2_{\LtwoLtwo} \right\}^{1/2}.
\end{equation}
Further, the solution space for the relative density is
\begin{equation} \label{def Xromean}
	\Xrho = \left \{\rorel \in H^1(0,T;H^{\frac{y}{2}}(\Omega))\cap L^\infty(0,T;H^{\frac{y+1}{2}}(\Omega)):\, \nabla\rorel \cdot \nu=0 \text{ on }\partial \Omega\right\},\ y>d-1,\ \, 2\leq y\leq 3,
\end{equation}
with the norm
\begin{equation}
	\begin{aligned}
		\|\rorel\|_{\Xrho} = \left\{ \|\rorel\|_{\LinfHyplusonehalf}^2+ \|\rorelt\|_{\LtwoHyhalf}\right\}^{1/2}.
	\end{aligned}
\end{equation}
The assumptions made on $y$ will be justified in the course of deriving energy estimates; see the discussion at the beginning of Section~\ref{sec energy estimate}. We note that the condition $y \leq 3$ can be removed if  $\gfunc \equiv 0$. The setting $\gfunc=\hfunc \equiv 0$ is considered in Section~\ref{sec vanishing viscosity}. \\
\indent Thirdly, as we will prove existence of the time-integrated pressure $\It p$, we introduce the corresponding solution space as
\begin{equation} \label{def XItp}
	\XItp = \left\{ \It p = \int_0^t p(s)\ds \in \LtwoTHone : \, \nabla p \cdot \nu = 0 \text{ on }\partial\Omega,\quad \frac{1}{|\Omega|}\int_{\Omega} p\dx=0 \right \}.
\end{equation}
 The targeted solution space for the studied problem is then $\solspacereg =  \Xumu \times \Xrho \times \XItp$. \\

\noindent{\bf Assumptions on data}. We assume that the source term satisfies
\begin{equation} \label{Xf}
	\bff\in \Xf=L^1(0,T;H_0(\text{div};\Omega))\cap L^2(0,T; \Ltwo),
\end{equation}
where $H_0(\text{div};\Omega)=\left\{\bfv\in L^2(\Omega) : \, \nabla\cdot\bfv=0 \text{ in }\Omega, \ 
\nu\cdot\bfv=0\text{ on }\partial\Omega \right\}$. The initial conditions are assumed to satisfy 
\[
(\bfuzero, \bfdzero, \rorelzero) \in \Hdiv \times \left(\Linf \cap \Hone\right) \times \Hyplusonetwo.
\]
 Additionally, we assume that 
 \begin{equation} \label{X BA}
 B/A \in \XBonA=\Linf \cap W^{1,3}(\Omega)
 \end{equation}
 and
\begin{equation}\label{Xromean}
	\begin{aligned}
		\romean \in \Xromean = \left\{v\in\Linf\ : \ \frac{1}{v}\in \Linf, \ \nabla\ln v\in \Linf\cap H^{\frac{y+1}{2}}(\Omega)\right\}
	\end{aligned}
\end{equation}
as well as that
\begin{equation}\label{Xcmean}
	\begin{aligned}
		\cmean^2 \in \Xcmean = \left \{v\in\Linf\cap W^{1,3}(\Omega)\ : \ \frac{1}{v}\in \Linf \right \}.
	\end{aligned}
\end{equation}
We next make precise what is meant by a solution of the problem.
\begin{definition}\label{def solution}We call  $(\bfu, \sigma, p) \in \solspacereg$ a solution of problem \eqref{regularized balance-state} supplemented with boundary \eqref{balance-state bc}  and initial conditions \eqref{balance-state ic} if it satisfies
	\begin{equation}
		\begin{aligned}
			\begin{multlined}[t]		\intTO \Bigl\{\left(\romean(\bfu-\bfuzero)+\nabla\It p - \mu \nabla (\divbfu) - \It \bff\right) \cdot \bfv 
				+\left( \rorelt + \afunc(\rorel) \divbfu - \gfunc(\bfu)\right) v 
				\\+	\left(\It p - \cmean^2 \romean \It (\bfunc(\rorel) \rorel)- \It \hfunc(\bfu) \right) \Deltaromean\phi+2\alpha_0 \bigl(\tau(-\Delta)^{\frac{y}{4}} (\rorel-\rorelzero) (-\Delta)^{\frac{y}{4}} \phi
				+\eta (-\Delta)^{\frac{y+1}{4}} \It \rorel (-\Delta)^{\frac{y+1}{4}} \phi \bigr) \Bigr\}  \dxt =0 \end{multlined}
		\end{aligned}
	\end{equation}
	for all $\bfv \in \LtwoTLtwod$, $v \in \LtwoTLtwo$, and $\phi \in L^2(0,T; \Hyplusonehalf)$, such that $\nabla \phi \cdot \nu =0$, with $\rorel_{\vert t=0} = \rorelzero$.
\end{definition}
The proof of existence of solutions is set up through a Faedo--Galerkin procedure. To this end, we first need to construct suitable approximations of $(\bfu, \rorel, p)$.

\subsection{Construction of Galerkin approximations} 

We approximate the system in \eqref{regularized balance-state} by constructing a Galerkin approximation of $(\rorel, p)$ by means of smooth eigenfunctions of the Neumann-Laplacian and then using it to set up suitable approximations of $\bfu$. This approach is in the spirit of Galerkin strategies for models of viscous compressible fluids; see~\cite{feireisl2006navier, feireisl2004dynamics, jungel2010global} and the references provided therein. However, here the relative density $\rorel$ and acoustic pressure $p$ are directly approximated by means of suitable basis functions as opposed to the velocity $\bfu$.\\
 \indent Let $\{w_i\}_{i \geq 1}$ be the eigenfunctions of the Neumann-Laplacian operator $-\Deltaromean$ acting on functions with zero mean, with eigenvalues $\{\lambda_i\}_{i \geq 1}$; that is, let
\begin{equation}
	\left \{
	\begin{aligned}
		&- \Deltaromean w_i=\lambda_i w_i \ &&\text{ in } \Omega, \\
		& \frac{1}{|\Omega|} \intO \wi \dx =0,&& \\
		&\nabla w_i \cdot \nu =0 \ &&\text{ on } \partial\Omega.
	\end{aligned} \right.
\end{equation}
 Fix $n \in \N$ and let $\Wn=\text{span} \{w_1, \ldots, w_n\}$.  
 We seek approximate $\rorel$ and $p$ in the form of
\begin{equation} \label{def Galerkin approximations}
	\begin{aligned}
		\sigman= \displaystyle \sum_{i=1}^n \xinsigma_i(t)w_i(x),\quad 		\pn=  \displaystyle \sum_{i=1}^n \xinp_i(t)w_i(x), 
	\end{aligned}
\end{equation}
with the unknown time-dependent coefficients $\xinsigma_i$, $\xinp_i:[0,T] \rightarrow \R$ for $i \in [1,n]$.  
Let the approximate initial relative density $\rorelnzero$  be the $\Hyplusonetwo$
projection of $\rorelzero$ on $\Wn$.
Denote $\boldsymbol{\xi}^n=[\xin_1 \, \ldots \, \xin_n]^T$ and $\boldsymbol{\xi}^n_0 = \boldsymbol{\xi}^n(0)$.\\
\indent  We then set $\bfun$ as the solution of the following system:
\begin{equation} \label{Galerkin approximate system}
\left\{	\begin{aligned}
	&\ \moG \quad &&\romean \bfun_{t} -\mu\nabla(\nabla \cdot \bfun)+ \nabla \pn =  \bff \quad &&\text{in } \Omega \times (0,T), \ &&\bfun(0) = \bfunzero,  \quad \nabla \bfun  \cdot \nu=0, \\[1mm]
		&\ \maG \quad   &&\rorelnt + \afunc(\roreln) \nabla\cdot\bfun -\gfunc(\bfun)=0 \quad  &&\text{in } \Wn \times (0,T), \ &&\roreln(0)= \rorelnzero, \\[1mm]
	&\ \pdG \quad &&\pn = \cmean^2 \romean \bfunc(\roreln)\roreln-\modL \roreln + \hfunc(\bfun) \quad  &&\text{in } \Wn \times (0,T),
	\end{aligned} \right.
\end{equation}
where $\gfunc(\bfun)=-\bfun \cdot \nabla \ln \romean$ and $\hfunc(\bfun)=\cmean^2 (\It \bfun + \bfdzero) \cdot \nabla \romean$; cf.\ \eqref{def gfunc hfunc}.  By considering $\maG$ and $\pdG$ in $\Wn$ we mean that we project them onto the finite dimensional space $\Wn$ with respect to the $L^2(\Omega)$ inner product. For showing that this approximation of \eqref{regularized balance-state} is well-posed, we need the following auxiliary existence result.
\begin{lemma} \label{lemma bfun}
Let $\mu>0$, $\romean,\,\romeaninv \in \Linf$ 
and $\bff \in \LtwoTLtwod$.  Let  $\bfunzero \in \Hdiv$.
Then, given $\pn \in L^2(0,T; \Wn)$, there exists a unique $\bfun \in \Xumu \cap \HoneTLtwod$ that satisfies
\begin{equation}\label{PDE_Lem1}
	\left\{	\begin{aligned}
		&  \bfun_{t} -\mu \romeaninv\nabla(\nabla \cdot \bfun) = \romeaninv (\bff-\nabla \pn),\\
		&  \bfun(0) = \bfunzero,  \quad
		\bfun \cdot \nu =0.
	\end{aligned} \right.
\end{equation}
\end{lemma}
\begin{proof}
We observe that the right-hand side satisfies $\romeaninv(\bff-\nabla \pn) \in \LtwoTLtwo$. The statement then follows along the lines of, e.g.,~\cite[Theorem 9.6]{salsa2022partial}; we omit the details here.
\end{proof}
Lemma~\ref{lemma bfun} allows us to define the solution operator $S: L^2(0,T; \Wn) \rightarrow \Xumu$, such that $S(\pn) = \bfun$. Let $p^{n, (1)}$, $p^{n, (2)} \in L^2(0,T; \Wn)$ and denote $\bfu^{n,(1)}=S(p^{n, (1)})$ and $\bfu^{n,(2)}=S(p^{n, (2)})$. By testing the problem solved by $\bfu^{n,(1)}-\bfu^{n,(2)}$ with $- \nabla (\nabla \cdot (\bfu^{n,(1)}-\bfu^{n,(2)}))$, we conclude that this operator is globally Lipschitz continuous:
\begin{equation} \label{Lipschitz continuity S}
\begin{aligned}	
	\|S(p^{n, (1)})-S(p^{n, (2)})\|_{\Xumu} 	
	=&\, \|\nabla \cdot(\bfu^{n,(1)}-\bfu^{n,(2)})\|_{\LinfLtwo}+ \sqrt{\mu} \| \nabla(\nabla \cdot(\bfu^{n,(1)}-\bfu^{n,(2)}))\|_{\LtwoLtwo}\\
		\leq&\,C_0 \|\nabla p^{n, (1)}-\nabla p^{n, (2)}\|_{\LtwoLtwo}\\
		\leq&\,C(n) \| p^{n, (1)}- p^{n, (2)}\|_{L^2(\Wn)},
\end{aligned}	
\end{equation}
where the last line follows by the equivalence of norms in finite-dimensional spaces.  The Galerkin problem then reduces to looking for a solution of
\begin{equation} \label{roreln subproblem}
\left\{	\begin{aligned}
		&\ \roreln = - \afunc(\roreln) \nabla\cdot S(\pn) +\gfunc(S(\pn)) \quad \text{ in } \Wn \times (0,T), \\ 
		&\ \roreln(0)= \rorelnzero,\\
		&\pn = \cmean^2 \romean \bfunc(\roreln)\roreln-\modL \roreln + \hfunc(S(\pn)) \quad \text{ in } \Wn \times (0,T),
	\end{aligned} \right.
\end{equation}
which we tackle in the next step. The solution is at first obtained on an $n$-dependent interval $[0, \Tn]$.
\begin{proposition} \label{prop roreln pn}
Let the assumptions of Lemma~\ref{lemma bfun} hold with $\romean \in \Xromean$, $B/A \in \XBonA$, and $\cmean^2 \in \Xcmean$. Then there exists $\Tn=T_n(n) \in (0,T)$, such that problem \eqref{roreln subproblem} has a unique solution $(\roreln, \pn) \in H^1(0,\Tn;\Wn) \cap L^2(0,\Tn; \Wn)$.
\end{proposition}
\begin{proof}
Let $R_1$, $R_2>0$. 
To prove unique solvability of \eqref{roreln subproblem}, we apply Banach's fixed-point theorem on the mapping 
\[
\calT: (\rorelnstar, \pnstar) \mapsto (\roreln, \pn),
\]
where $ (\rorelnstar, \pnstar)$ is taken from the ball
\begin{equation}
	\begin{aligned}
	 \ball = \Bigl\{(\rorelnstar, \pnstar )\in H^1(0,T; \Wn) \times L^2(0,T; \Wn): &\, \|\rorelnstar\|_{H^1(0,T; \Ltwo)} \leq R_1, \
		\|\pnstar\|_{\LtwoLtwo} \leq R_2, \\
		&\, \rorelnstar (0) = \rorelnzero \ \Bigr \},
	\end{aligned}
\end{equation}
 and $(\roreln, \pn)$ solves the following linear problem:
\begin{equation} \label{roreln subproblem linearized}
	\left\{	\begin{aligned}
		&\ \rorelnt = - \afunc(\rorelnstar) \nabla\cdot S(\pnstar) + \gfunc( S(\pnstar)) \quad \text{ in } \Wn \times (0,T), \\ 
		&\ \roreln(0)= \rorelnzero,\\
		&\pn+\modL \roreln = \cmean^2 \romean \bfunc(\rorelnstar)\rorelnstar + \hfunc( S(\pnstar))\quad \text{ in } \Wn \times (0,T).
	\end{aligned} \right.
\end{equation}

\noindent \underline{Self-mapping}:  Take $(\rorelnstar, \pnstar) \in \ball$. We first check that $(\roreln, \pn)=\calT (\rorelnstar, \pnstar) \in \ball$. Note that
\begin{equation}
	\begin{aligned}
		\|\rorelnt\|_{\HoneLtwo} \leq&\, \|\rorelnt\|_{\LtwoLtwo} + \|\It \rorelnt + \rorelnzero\|_{\LtwoLtwo}\\
		\leq&\, (1+T)\|\rorelnt\|_{\LtwoLtwo} + \sqrt{T}\|\rorelnzero\|_{\Ltwo}.
	\end{aligned}
\end{equation}
Using the first equation in \eqref{roreln subproblem linearized}, we then have
\begin{equation} \label{est roreln}
	\begin{aligned}
	\|\roreln\|_{H^1(\Ltwo)}
	 \leq&\,\begin{multlined}[t] (1+T)\left( \|\afunc(\rorelnstar) \nabla\cdot S(\pnstar)\|_{L^2(\Ltwo)} +\|\gfunc(S(\pnstar))\|_{L^2(\Ltwo)} \right)+ \sqrt{T}\|\rorelnzero\|_{\Ltwo} 
			\end{multlined}
			\\
				 \leq&\, \begin{multlined}[t] (1+T)\sqrt{T}\left( \|\afunc(\rorelnstar)\|_{\Linf} \|\nabla\cdot S(\pnstar)\|_{L^\infty(\Ltwo)} +\|\gfunc(S(\pnstar))\|_{L^\infty(\Ltwo)} \right)+ \sqrt{T}\|\rorelnzero\|_{\Ltwo} .
			\end{multlined}
	\end{aligned}
\end{equation}
By relying on the estimate
\begin{equation}
\begin{aligned}
\|\gfunc(S(\pnstar))\|_{L^2(\Ltwo)}=\|-S(\pnstar) \cdot \nabla \ln \romean\|_{L^2(\Ltwo)} \leq&\, \|\nabla\ln \romean\|_{\Linf}\sqrt{T} \|S(\pnstar)\|_{L^\infty(\Ltwo)}
\end{aligned}
\end{equation}
and the equivalence of norms in finite-dimensional spaces, from \eqref{est roreln}, we conclude that
\begin{equation} \label{est roreln new}
	\begin{aligned}
		\|\roreln\|_{H^1(\Ltwo)}
		\leq&\, C(n)(1+T)\sqrt{T}\left((1+R_1) R_2 + \|\nabla \ln \romean\|_{\Linf}  R_2 \right) +\sqrt{T} \|\rorelnzero\|_{\Ltwo}.
	\end{aligned}
\end{equation}
We can thus guarantee that $\|\roreln\|_{H^1(\Ltwo)} \leq R_1$ by reducing $T=T(n)$. \\
\indent From the last equation in \eqref{roreln subproblem linearized} and the fact that $\|\modL (\roreln)\|_{\LtwoLtwo} \leq C(n) \|\roreln\|_{\HoneLtwo}$, we can estimate $\pn$ as follows:
\begin{equation}  \label{est pn}
	\begin{aligned}
		\|\pn\|_{\LtwoLtwo} \leq&\, \sqrt{T}\|\cmean^2 \romean \bfunc(\rorelnstar)\roreln\|_{\LinfLtwo} +\|L\roreln\|_{\LtwoLtwo}+ \sqrt{T}\|\hfunc(S(\pnstar))\|_{\LinfLtwo} \\
		\leq&\, \begin{multlined}[t] C(n)\left( \sqrt{T}\|\cmean^2\romean\|_{\Linf} (1+ \tfrac12 \|B/A\|_{\Linf})R_1+\|\roreln\|_{\HoneLtwo} +\sqrt{T} \| \cmean^2\nabla \romean\|_{\Linf}  R_2 \right)\\+\sqrt{T} \| \cmean^2\nabla \romean\|_{\Linf} \|\bfdzero\|_{\Ltwo},
			\end{multlined}
	\end{aligned}
\end{equation}
where we have used the fact that
\begin{equation}
	\begin{aligned}
 \|\hfunc(S(\pnstar))\|_{\LinfLtwo} =&\, \|\cmean^2\left(\It(S(\pnstar)) +\boldsymbol{d}_0\right)\cdot \nabla \romean\|_{\LinfLtwo}\\
  \leq&\, \|\cmean^2 \nabla \romean\|_{\Linf}T\|S(\pnstar)\|_{\LinfLtwo}+\|\cmean^2 \bfdzero \cdot \nabla \romean\|_{\Ltwo}.
 \end{aligned}
\end{equation}
Since we can reduce $\|\roreln\|_{\HoneLtwo}$ by reducing the final time, from \eqref{est pn}, we conclude that
\[
	\|\pn\|_{\LinfLtwo}  \leq R_2,
\]
provided $T=T(n)$ is small enough. \\[0mm]

\noindent \underline{Contractivity}: Let $(\rorelnstarone, \pnstarone)$, $(\rorelnstartwo, \pnstartwo) \in \ball$ and denote $(\rorelnone, \pnone)=\calT (\rorelnstarone, \pnstarone)$ and $(\rorelntwo, \pntwo)= \calT(\rorelnstartwo, \pnstartwo)$. Further, we introduce the following notation for the differences:
\begin{equation}
	\begin{aligned}
		\ororelnstar =&\, \rorelnstarone-\rorelnstartwo, \quad &&\ororeln= \rorelnone-\rorelntwo, \\
		\opnstar=&\,\pnstarone-\pnstartwo, \ &&\opn=\pnone-\pntwo.
	\end{aligned}
\end{equation}
We can see  $(\ororeln, \opn)$ as the solution to the following problem:
\begin{equation} \label{diff roreln subproblem linearized}
	\left\{	\begin{aligned}
		&\ \ororeln_t = - \afunc(\rorelnstarone) \nabla\cdot (S(\pnstarone)-S(\pnstartwo))- 2\ororelnstar\nabla\cdot S(\pnstartwo)+ \gfunc\left( S(\pnstarone)- S(\pnstartwo)\right) \quad \text{ in } \Wn \times (0,T), \\ 
		&\ \ororeln(0)= 0,\\
		&\opn +\modL \ororeln=\cmean^2 \romean \frac{B}{2A}\ororelnstar\rorelnstarone+ \cmean^2 \romean \bfunc(\rorelnstartwo)\ororelnstar + \hfunc\left( S(\pnstarone)-S(\pnstartwo)\right)\quad \text{ in } \Wn \times (0,T),
	\end{aligned} \right.
\end{equation}
where we have used the fact that $a(\rorelnstarone)-a(\rorelnstartwo) = 2 \ororelnstar$ and $b(\rorelnstarone)-b(\rorelnstartwo)= \frac{B}{2A}\ororelnstar$; cf.\ \eqref{def a b}.
Similarly to \eqref{est roreln}, we then have the following estimate:
\begin{equation}
	\begin{aligned}
		\|\ororeln\|_{\HoneLtwo} 
		\leq&\, \begin{multlined}[t]  (1+T)\sqrt{T}\Bigl(\|\afunc(\rorelnstarone)\nabla \cdot (S(\pnstarone)-S(\pnstartwo)) \|_{\LinfLtwo} + 2\|\ororelnstar\nabla \cdot S(\pnstartwo)\|_{\LinfLtwo}\\+ \|\gfunc\left(S(\pnstarone)-S(\pnstartwo)\right)\|_{\LinfLtwo}.	\end{multlined} 
	\end{aligned}
\end{equation}
By relying on the fact that
\[
\begin{aligned}
\|\afunc(\rorelnstarone)\|_{\LinfLinf} \leq&\, C(n)(1+R_1), \\
 \|\nabla \cdot S(\pnstartwo)\|_{\LinfLinf}\leq&\, C(n) \|\nabla \cdot S(\pnstartwo)\|_{\LinfLtwo} \leq\, C(n) \|\pnstartwo\|_{\LtwoLtwo} \leq\, C(n) R_2,
\end{aligned}
\]
together with the Lipschitz continuity of $S$ (see \eqref{Lipschitz continuity S})  and
\[
\begin{aligned}
	\|\gfunc\left(S(\pnstarone)-S(\pnstartwo)\right)\|_{\LinfLtwo} \leq C(n) \|\nabla \ln \romean\|_{\Linf}\|\opnstar\|_{\LtwoLtwo},
\end{aligned}
\]
we obtain
\begin{equation} \label{est ororeln}
	\begin{aligned}
		\|\ororeln\|_{\HoneLtwo} 
		\lesssim&\, C(n)(1+T) \sqrt{T} \left(\|\ororelnstar\|_{\LinfLtwo}+ \|\opnstar\|_{\LtwoLtwo}\right).
	\end{aligned}
\end{equation}
We can bound the differences of pressures as follows:
\begin{equation} \label{est 1 opn}
	\begin{aligned}
		\|\opn\|_{\LtwoLtwo} \leq&\, \begin{multlined}[t] \sqrt{T}\|\cmean^2\romean \tfrac{B}{2A} \ororelnstar \rorelnstarone\|_{\LinfLtwo} +\sqrt{T}\|\cmean^2\romean \bfunc(\rorelnstartwo) \ororelnstar\|_{\LinfLtwo}
			+\|\modL (\ororeln)\|_{\LtwoLtwo}\\+ \|\hfunc(S(\pnstarone)-S(\pnstartwo))\|_{\LtwoLtwo}.
			\end{multlined}
	\end{aligned}
\end{equation}
By the equivalence of norms in finite-dimensional spaces and estimate \eqref{est ororeln}, we infer
\begin{equation}
	\begin{aligned}
	\|\modL (\ororeln)\|_{\LtwoLtwo} \leq C(n) \|\ororeln\|_{\HoneLtwo} \lesssim C(n) \sqrt{T} \left(\|\ororelnstar\|_{\LinfLtwo}+ \|\opnstar\|_{\LtwoLtwo}\right).
	\end{aligned}
\end{equation}
Further,
\begin{equation}
	\begin{aligned}
		 \|\hfunc(S(\pnstarone)-S(\pnstartwo))\|_{\LtwoLtwo} \leq&\, \sqrt{T} \|\cmean^2 \It \left(S(\pnstarone)-S(\pnstartwo)\right)\cdot \nabla \romean\|_{\LinfLtwo} \\
		 \leq&\,  \sqrt{T}\|\cmean^2 \nabla \romean\|_{\Linf}\|S(\pnstarone)-S(\pnstartwo\|_{\LinfLtwo}\\
		 \leq&\, C(n) \sqrt{T} \|\opnstar\|_{\LtwoLtwo}.
	\end{aligned}
\end{equation}
Using this bound in \eqref{est 1 opn} together with the Lipschitz continuity of the operator $S$ yields
\begin{equation} \label{est opn}
	\begin{aligned}
		\|\opn\|_{\LtwoLtwo} \leq&\, C(n) \sqrt{T}\left(\|\ororelnstar\|_{\HoneLtwo}+ \|\opnstar\|_{\LtwoLtwo}\right).
	\end{aligned}
\end{equation}
By adding the two bounds, \eqref{est ororeln} and \eqref{est opn}, we arrive at
\begin{equation}
	\|\ororeln\|_{\HoneLtwo} +\|\opn\|_{\LtwoLtwo} \leq C(n)(1+T)\sqrt{T}\left(\|\ororelnstar\|_{\LinfLtwo}+ \|\opnstar\|_{\LtwoLtwo}\right).
\end{equation}
Thus strict contractivity of the mapping can be guaranteed by reducing $T=T(n)$. An application of Banach's fixed-point theorem yields the statement.
\end{proof}


\subsection{Energy identity for Galerkin approximations}
Having constructed Galerkin approximations, in the next step, we derive an energy identity for \eqref{Galerkin approximate system} on $[0,\Tn]$.  For this purpose, we introduce $\gcoefprojection=\projection[\gfunc(\bfun)]\in\Wn$ as the Ritz projection of $\gcoef=\gfunc(\bfun)= -\bfun\cdot \nabla \ln \romean$ in the sense of
\begin{equation}\label{GalProj}
	\intO \frac{1}{\romean}\nabla \gfunc(\bfun) \cdot\nabla v_n \dx=\intO \frac{1}{\romean}\nabla \gcoefprojection \cdot\nabla v_n \dx\quad \ \text{ for all} \ v_n \in \Wn;
\end{equation}
that is, 
\[
(-\Deltaromean \gfunc(\bfun),\, v_n)_{L^2} = ( -\Deltaromean \gcoefprojection,\, v_n)_{L^2}\quad \ \text{ for all} \ v_n \in \Wn.
\]
In the derivation of the energy identity for $(\bfun, \roreln, \pn)$, we rely on the stability of this projection operator in the following sense.
\begin{lemma}\label{lemGprojg}
For $\gcoef= \gfunc(\bfun)=-\nabla\ln\romean\cdot\bfun$, where $\bfun\in \Linf \cap \Hyplusonetwo$, $\romean\in\Xromean$, $\|\nabla\ln\romean\|_{\Linf\cap\Hyplusonetwo}\leq1$, we have 
\begin{equation} \label{stability bounds grpojection}
	\begin{aligned}
		&\|\nabla \gcoefprojection\|_{\Ltwo} 
\leq  C \|\nabla\ln\romean\|_{\Linf\cap W^{1,3}(\Omega)} \|\bfun\|_{\Hone},	\\[1mm]
		&	\|(-\Delta_N)^{\frac{y}{4}}\gcoefprojection\|_{\Ltwo} 
\leq C \|\nabla\ln\romean\|_{\Linf\cap \Hytwo} \|\bfun\|_{\Linf\cap \Hytwo},\\[1mm]
		&\|(-\Delta_N)^{\frac{y+1}{4}}\gcoefprojection\|_{\Ltwo} 
\leq C  \|\nabla\ln\romean\|_{\Linf\cap\Hyplusonetwo} \|\bfun\|_{\Linf\cap\Hyplusonetwo},
	\end{aligned}
\end{equation}
with $C$ depending only on $\|\romean\|_{\Linf}$, $\|\romeaninv\|_{\Linf}$, but not on $n$.
\end{lemma}
\begin{proof}
The proof is provided in the appendix.
\end{proof}

We proceed to derive an energy identity for $(\bfun, \roreln, \pn)$ on $[0, \Tn]$ under the assumption of uniform smallness of solutions on $[0, \Tn]$.
\begin{proposition} \label{prop: Galerkin energy id}
	Let the assumptions of Lemma~\ref{lemma bfun} and Proposition~\ref{prop roreln pn} hold with $\bff \in \Xf$. Let $(\bfun, \roreln, \pn)$ be the solution of \eqref{Galerkin approximate system} on $[0, \Tn]$. Assume that there exists $r>0$, independent of $n$, such that
	\begin{equation} \label{r Linf smallness}
		|\roreln(x,t)| \leq r \quad \text{for all } (x,t) \in \Wn \times [0, \Tn].
	\end{equation}
	Then if $r>0$ is sufficiently small, there exist $\ulacoef$, $\olacoef>0$ and $\ulbfunc$, $\olbfunc>0$, independent of $n$, such that
	\begin{equation} \label{nondegeneracy}
		\begin{aligned}
			&	0 < \ulafunc \leq	\afunc(\roreln) \leq \olafunc \ \quad &&\text{for all } (x,t) \in \Wn \times [0, \Tn],\\
			& 0 <  \ulbfunc \leq	\bfunc(\roreln) \leq  \olbfunc \ \quad &&\text{for all } (x,t) \in \Wn \times [0, \Tn],
		\end{aligned}
	\end{equation}	
and the following identity holds:
	\begin{equation}\label{energy identity}
		\begin{aligned}
			&\begin{multlined}[t]\frac12 \ddt \Bigl(\|
				{\sqrt{\afunc(\roreln)}}\nabla\cdot\bfun\|_{\Ltwo}^2 
				+ \|\cmean\sqrt{\bfunc(\roreln)}\, \nabla \roreln\|_{\Ltwo}^2\Bigr)
				+\mu\|\sqrt{(\afunc(\roreln))/\romean}\nabla(\nabla\cdot\bfun)\|_{\Ltwo}^2\\
				+\alpha_0\Bigl(2\tau\|(-\Delta_N)^{\frac{y}{4}} \rorelnt\|_{\Ltwo}^2
				+\eta \ddt  \|(-\Delta_N)^{\frac{y+1}{4}} \roreln\|_{\Ltwo}^2\Bigr) 
				= 
				\rhsone+\rhstwo,\end{multlined}
		\end{aligned} 
	\end{equation}
	where the right-hand side terms are given by
	\begin{equation} \label{rhsone}
		\begin{aligned}
			\rhsone=&\, \begin{multlined}[t]
				-\intO	{\afunc(\roreln)\nabla \cdot(\romean^{-1} \bff)\nabla\cdot\bfun \dx}	
				+ \intO \frac12\cmean^2 
\bfunc'(\roreln)\rorelnt
|\nabla\roreln|^2 \dx
				\\	- \intO \roreln\left(\nabla [\cmean^2 \bfunc(\roreln)]+\cmean^2 \bfunc(\roreln)\nabla\ln\romean\right)\cdot\nabla\rorelnt \dx 
				+\frac12 \intO 
\afunc'(\roreln)\rorelnt
|\nabla \cdot \bfun|^2  
				\\	+\mu\intO \frac{1}{\romean}\nabla(\nabla\cdot\bfun)\cdot (\afunc'(\roreln) \nabla \roreln) \divbfun \dx
			\end{multlined}
		\end{aligned}
	\end{equation}
	and
	\begin{equation} \label{rhstwo}
		\begin{aligned}
			\rhstwo=&\, \begin{multlined}[t]
				-\intO \frac{1}{\romean}\nabla \hfunc(\bfun)\cdot\nabla \rorelnt \dx  \\+	\intO \Bigl(
				\frac{1}{\romean}\nabla \gprojection(\bfun) \cdot\nabla [\cmean^2\romean\bfunc(\roreln)\,\roreln-\hfunc(\bfun)]
				-2\alpha_0 \bigl(\tau(-\Delta_N)^{\frac{y}{2}} \rorelnt 
				+\eta (-\Delta_N)^{\frac{y+1}{2}} \roreln\bigr)\gprojection(\bfun)
				\Bigr)\dx
			\end{multlined}
		\end{aligned}
	\end{equation}
with 
$\afunc'(\roreln)=2$ and $\bfunc'(\roreln)=\frac{B}{2A}$.
\end{proposition}	
\begin{proof}
Since $\afunc(\roreln)=1+2\roreln$ and $\bfunc(\roreln)=1+\frac{B}{2A} \roreln$, the bounds in \eqref{nondegeneracy} follow immediately by \eqref{r Linf smallness} if $r$ is small enough. The identity in \eqref{energy identity} is obtained by convenient testing of the problem that will lead to cancellations of several terms. We test equation $\moG$\ in \eqref{Galerkin approximate system} with 
	\[
	\bfvn= - \frac{1}{\romean}\nabla (\afunc(\roreln(t)) \nabla \cdot \bfun(t)),
	\]
	equation $\maG$  with $-\Deltaromean \pn(t)$, and equation $\pdG$\  with $\Deltaromean \rorelnt(t)$. We note that we are allowed to do this because $\bfvn \in \Ltwod$ and $-\Deltaromean \pn(t)$, $ \Deltaromean \rorelnt(t)\in \Wn$. Proceeding in this manner, integrating over $\Omega$, and integrating by parts in space yields
	
	\begin{equation}\label{var regularized tested}
		\begin{aligned}
			\begin{multlined}[t]\intO -(\romean \bfun_t+\nabla \pn-\mu\nabla(\nabla\cdot  \bfun)- \bff\bigr)\cdot
				\frac{1}{\romean}\nabla(\afunc(\roreln)\nabla\cdot\bfun) \dx
				+\intO \nabla \bigl( \rorelnt + \afunc(\roreln)\nabla\cdot\bfun -\gfunc(\bfun)\bigr)\cdot
				\frac{1}{\romean}\nabla \pn \dx \\
				\qquad-\intO \nabla \bigl(\pn-\cmean^2\romean\bfunc(\roreln)\,\roreln-\hfunc(\bfun)\bigr)\cdot
				\frac{1}{\romean}\nabla \rorelnt \dx
				+2\alpha_0\intO (-\Deltaromean)^{-1}\bigl(\tau(-\Delta)^{\frac{y}{2}} \rorelnt 
				+\eta (-\Delta)^{\frac{y+1}{2}} \roreln\bigr)(-\Deltaromean \rorelnt)\dx 
				=0 \end{multlined}
		\end{aligned}
	\end{equation}
	a.e. in time. Conveniently, the (space-integrated) terms $-\nabla \pn \cdot  \frac{1}{\romean}\nabla(\afunc(\roreln)\nabla\cdot\bfun)$  and $\nabla(\afunc(\roreln)\nabla\cdot\bfun) \cdot  \frac{1}{\romean}\nabla \pn$ as well as $  \nabla \rorelnt \cdot  \frac{1}{\romean}\nabla \pn$ and $-\nabla \pn \cdot  \frac{1}{\romean}\nabla \rorelnt$ cancel out and we are left with
	\begin{equation}\label{var regularized tested simplified}
		\begin{aligned}
			\begin{multlined}[t]\intO -(\romean\bfun_t-\mu\nabla(\nabla\cdot\bfun)-\bff\bigr)\cdot
				\frac{1}{\romean}\nabla(\afunc(\roreln)\nabla\cdot\bfun) \dx
			-\intO \nabla \gfunc(\bfun) \cdot
				\frac{1}{\romean}\nabla \pn \dx \\
				\qquad-\intO \nabla \bigl(-\cmean^2\romean
\bfunc(\roreln)
\,\roreln-\hfunc(\bfun)\bigr)\cdot
				\frac{1}{\romean}\nabla \rorelnt \dx
				+2\alpha_0\intO (\tau(-\Delta)^{\frac{y}{2}} \rorelnt 
				+\eta (-\Delta)^{\frac{y+1}{2}} \roreln)\rorelnt\dx  =0. \end{multlined}
		\end{aligned}
	\end{equation}
	To transform the terms further, we can employ the following identities:
	\[
	\begin{aligned}
		-\int_\Omega
		\bfun_t \cdot  \nabla(\afunc(\roreln)\nabla\cdot\bfun) \dx
		=& \intO
		\afunc(\roreln) \nabla \cdot \bfun_t   \nabla\cdot\bfun \dx -\int_{\partial\Omega} (\afunc(\roreln)\nabla\cdot\bfun)
		\bfun_t \cdot\nu \dS\\
		=&\, \int_\Omega 
		\frac12\ddt |{\sqrt{\afunc(\roreln)}}\nabla\cdot\bfun|^2
		\dx -\frac12 \intO 
\afunc'(\roreln)\rorelnt
|\nabla \cdot \bfun|^2 \dx
	\end{aligned}
	\]
and
	\begin{equation}
		\begin{aligned}
			\mu	\intO  \nabla(\nabla\cdot\bfun)\cdot\frac{1}{\romean}\nabla(\afunc(\roreln)\divbfun) \dx = \mu \| \sqrt{\afunc(\roreln)/{\romean}}  \nabla(\nabla\cdot\bfun)\|^2_{\Ltwo}+\mu\intO \nabla(\divbfun)\cdot \frac{1}{\romean}\nabla \afunc(\roreln) \divbfun \dx,
		\end{aligned}
	\end{equation}
	as well as, with $\beta=\cmean^2\bfunc(\roreln)$,
	\[
	\nabla[\romean\beta\roreln]\cdot\frac{1}{\romean}\nabla \roreln_t
	=\frac12\ddt |\sqrt{\beta}\nabla\roreln|^2 -\frac12\beta_t|\nabla\roreln|^2
	+\roreln(\nabla\beta+\beta\nabla\ln\romean)\cdot\nabla\rorelnt,
	\]
where $\beta_t = \cmean^2 \frac{B}{2A}\rorelnt$.	In this way we obtain the energy identity
	\begin{equation}\label{enid2 reg}
		\begin{aligned}
			&\begin{multlined}[t]\frac12 \ddt \Bigl(\|
				{\sqrt{\afunc(\roreln)}}\nabla\cdot\bfun\|_{\Ltwo}^2 
				+ \|\cmean\sqrt{\bfunc(\roreln)}\, \nabla \roreln\|_{\Ltwo}^2\Bigr)
				+\mu\|\sqrt{{\afunc(\roreln)/\romean}}\nabla(\nabla\cdot\bfun)\|_{\Ltwo}^2\\
				+\alpha_0\Bigl(2\tau\|(-\Delta)^{\frac{y}{4}} \rorelnt\|_{\Ltwo}^2
				+\eta \ddt  \|(-\Delta)^{\frac{y+1}{4}} \roreln\|_{\Ltwo}^2\Bigr) \end{multlined}
			\\
			&= \begin{multlined}[t]-\intO 
				\frac{1}{\romean}\bff\cdot\nabla(\afunc(\roreln)\nabla\cdot\bfun) \dx
				+\intO \frac{1}{\romean}\nabla \gfunc(\bfun) \cdot\nabla \pn \dx
				- \intO \frac{1}{\romean}\nabla \hfunc(\bfun)\cdot\nabla \rorelnt \dx \\
				+ \intO \frac12\cmean^2 
\bfunc'(\roreln)\rorelnt
|\nabla\roreln|^2 \dx
				- \intO \roreln(\nabla [\cmean^2 \bfunc(\roreln)]+\cmean^2 \bfunc(\roreln)\nabla\ln\romean)\cdot\nabla\rorelnt \dx \\
				+\frac12 \intO 
\afunc'(\roreln)\rorelnt
|\nabla \cdot \bfun|^2 \dx
				+\mu\intO \frac{1}{\romean}\nabla(\nabla\cdot\bfun)\cdot \nabla \afunc(\roreln) \divbfun \dx:=\rhs.
			\end{multlined}
		\end{aligned} 
	\end{equation}
	Note that the term with $\nabla \pn$ on the right-hand side of \eqref{enid2 reg} cannot be controlled directly by the left-hand side terms so we would not be able to derive an energy estimate starting from \eqref{enid2 reg}. To mend this, we rewrite this term by additionally testing equation $\pdG$\ in \eqref{Galerkin approximate system} with $-\Deltaromean \gprojection(\bfun) \in \Wn$.
	We then have  
	\begin{equation}\label{nablap reg}
		\begin{aligned}
			\intO \frac{1}{\romean}\nabla \gfunc(\bfun) \cdot\nabla \pn \dx
			=&\, \intO \frac{1}{\romean}\nabla \gprojection(\bfun)  \cdot\nabla \pn \dx\\
			=&\,  \intO \Bigl(
			\frac{1}{\romean}\nabla \gprojection(\bfun) \cdot\nabla [\cmean^2\romean\bfunc(\roreln)\,\roreln-\hfunc(\bfun)]
			-2\alpha_0 \bigl(\tau(-\Delta)^{\frac{y}{2}} \rorelnt 
			+\eta (-\Delta)^{\frac{y+1}{2}} \roreln\bigr)\gprojection(\bfun)
			\Bigr)\dx.
		\end{aligned}
	\end{equation}
Using this identity, the right-hand side of \eqref{enid2 reg} can be rewritten as the sum $\rhs=\rhsone+\rhstwo$, where $\rhsone$ is defined in \eqref{rhsone} and $\rhstwo$ in \eqref{rhstwo}, to arrive at the claim. 
\end{proof}
\subsection{Energy estimate} \label{sec energy estimate}
Starting from the obtained identity in \eqref{energy identity}, we next derive an energy estimate, at first, on $[0, \Tn]$ and again under an assumption of uniform smallness of solutions. Concerning the regularity induced by the $y$-power damping terms on the left-hand side of \eqref{enid2 reg}, there are several requirements that we needed to take into account:
\begin{itemize}
	\item First of all, we need to obtain a bound on $\rorel$ from the $\eta$ term in \eqref{enid2 reg} whose control in its turn enables non-degeneracy of $\afunc(\roreln)=1+2 \roreln$ and $\bfunc(\roreln)=1+\frac{B}{2A} \roreln$. Thus we require that $2\tfrac{y+1}{4}>\frac{d}{2}$. 
	\item Secondly, $\rhsone$, given in \eqref{rhsone}, contains the gradient of $\rorel_t$, which we have to control by the left-hand side term $2 \alpha_0 \tau \|(-\Delta_N)^{\frac{y}{4}} \rorelnt\|^2_{\Ltwo}$ in \eqref{enid2 reg}, resulting in the requirement $2\tfrac{y}{4}\geq1$.
	\item    Thirdly, to be able to absorb the $\gfunc(\bfun)$ terms in $\rhstwo$, defined in \eqref{rhstwo}, by the left-hand side, we need an upper bound on $y$: $y \leq 3$.
\end{itemize}  
  Altogether, we thus assume that
\begin{equation} \label{assumptions y}
	y>d-1 \text{ and } 2 \leq y\leq 3. 
\end{equation}
As mentioned before, the condition $y \leq 3$ can be removed if $\gfunc \equiv 0$. The case $\gfunc=\hfunc \equiv 0$ is analysed in Section~\ref{sec vanishing viscosity} in a $\mu$-uniform manner for which the lower bound on $y$ has to be strengthened, however. \\
\indent In the analysis below, we use the Poincar\'{e}--Friedrichs inequality as well as elliptic regularity of the Neumann problem \cite[Theorem 4, p 217]{Mikhajlov1978}
to conclude existence of constants $C_s$, $\tilde{C}_s$, such that 
\[
\|\phi\|_{H^s(\Omega)}\leq
C_s \|(-\Delta_N)^{s/2} \phi\|_{L^2(\Omega)} 
\leq \tilde{C}_s \|\phi\|_{H^s(\Omega)}
\text{ for all }\phi\in H^s(\Omega), \ \intO \phi =0, \quad s\in  \left\{\frac{y}{2},\frac{y+1}{2}\right\}.
\]
We note that, under the assumptions \eqref{assumptions y} made on $y$, we have continuity of the embeddings 
\begin{equation}\label{embeddings_y}
	H^{\frac{y}{2}}(\Omega)\to \Hone \to \Lsix, \quad
	H^{\frac{y+1}{2}}(\Omega)\to \Linf\cap \Wonethree
\end{equation} 
for $d\in\{2,3\}$.
We next derive a uniform bound for the sum of the semi-discrete energy and dissipation functionals at time $t$ given by
\begin{equation}
	\begin{aligned}
		\calE(t) = \|\bfun(t)\|^2_{\Hdiv} 
		+ \|\nabla \roreln(t)\|^2_{\Ltwo} + \|\roreln(t) \|^2_{\Hyplusonetwo}
	\end{aligned}
\end{equation}
and
\begin{equation}
	\begin{aligned}
		\calD(t) = \intt  \left(\mu \|\nabla(\nabla\cdot\bfun(s))\|^2_{\Ltwo} + \|\rorelnt(s)\|^2_{\Hytwo}+ \|\nabla \Is p\|^2_{\Ltwo} \right)\ds,
	\end{aligned}
\end{equation}
at first, for $t \in [0, \Tn]$. In the subsequent step, we will use this result to bootstrap the existence and the energy bounds to $[0,T]$.
\begin{proposition} \label{prop: Galerkin energy estimate} Let the assumptions of Proposition~\ref{prop: Galerkin energy id} hold and let the approximate initial velocity $\bfunzero \in \Hdiv$ satisfy
	\[
	\bfunzero \rightarrow \bfuzero \ \text{ in } \Hdiv \quad \text{ as } \  n \rightarrow \infty.
	\]
 Let 
the condition \eqref{assumptions y} on $y$ as well as 
			\begin{equation}\label{smallness condition romean cmean}
		\begin{aligned}
			\begin{multlined}[t]
				\|\nabla [\cmean^2 \nabla \romean] \|_{\Ltwo}+\|\cmean^2 \nabla \romean\|_{\Lthree}+\|\nabla \ln \romean\|_{\Hyplusonetwo}< \delta_{\romean, \cmean}
			\end{multlined}
		\end{aligned}
	\end{equation}
hold. 
	Furthermore, assume that there exist $r>0$, independent of $n$, such that
	\begin{equation}\label{smallness condition acoef bcoef}
		\begin{aligned}
			\begin{multlined}[t]
				\|\roreln\|_{L^\infty(0,\Tn; \Linf)}+	
{\|\afunc'(\roreln)\rorelnt\|_{\LtwonLthree}}
+\| \afunc'(\roreln) \nabla \roreln\|_{\LinfnLthree} +
{\|\cmean^2 \bfunc'(\roreln)\rorelnt\|_{\LtwonLsix}}
\\
				+\|\nabla [\cmean^2 \bfunc(\roreln)]\|_{\LinfnLtwo}+\|\cmean^2 \bfunc(\roreln)\nabla\ln\romean\|_{\LinfnLtwo}< r
			\end{multlined}
		\end{aligned}
	\end{equation}
with 
$\afunc'(\roreln)=2$ and $\bfunc'(\roreln)=\frac{B}{2A}$.	Then for sufficiently small $r$ and sufficiently small $\deltaromeancmean$, independently of $n$, the following bound holds: 
	\begin{equation}\label{uniform energy bound Galerkin}
		\begin{aligned}
			&\esssup_{t \in (0,\Tn)}\calE(t) + \esssup_{t \in (0,\Tn)}\calD(t)\\
			\leq&\, \begin{multlined}[t] C_1(\romean) \exp(C_2\Tn) \Bigl(\|\nabla \cdot (\romean^{-1} \bff)\|^2_{\LtwoLtwo}+\|\It \bff\|^2_{\LtwoLtwo} 
				+ \|\bfuzero\|^2_{\Hdiv} 
				+ \|\cmean\sqrt{\rorelzero}\, \nabla \rorelzero\|^2_{\Ltwo}
					+\|\rorelzero \|^2_{\Hyplusonetwo}\\ +\| \nabla [\cmean^2\bfdzero \cdot \nabla \romean] \|^2_{\Ltwo}	\Bigr), \end{multlined}
		\end{aligned}
	\end{equation}
	where $C_1$ and $C_2$ do not depend on $\Tn$ or $n$. 
\end{proposition}	
Note that the smallness assumption on the gradients of $\cmean$ and $\romean$ made in \eqref{smallness condition romean cmean} only restricts their variations but still allows for large absolute values of these quantities. 
\begin{proof}
	We start from the derived energy identity in \eqref{enid2 reg} and estimate the right-hand side terms within $\rhsone$ and $\rhstwo$. \\
	
\noindent \underline{Estimate of $\rhsone$}:	The time integral of the first right-hand side term $\rhsone$ can be bounded using H\"older's inequality and the fact that $a'(\roreln)=2$ as follows: 
	\begin{equation}\label{rhsone I}
		\begin{aligned}
			\int_0^t \rhsone(s)\ds
			\leq&\, \begin{multlined}[t] 
				\|\nabla \cdot (\romean^{-1} \bff)\|_{\LoneLtwo} \olacoef \|\divbfun \|_{\LinftLtwo}
				+ 
				\frac12\|\cmean^2 \bfunc'(\roreln)\rorelnt\|_{\LtwotLsix}\|\nabla\roreln\|_{\LtwotLthree}\|\nabla\roreln\|_{\LinftLtwo}
				\\	+\|\roreln\|_{\LtwotLinf} \|\nabla\rorelnt\|_{\LtwotLtwo} \left( \|\nabla [\cmean^2 \bfunc(\roreln)]\|_{\LinftLtwo}+\|\cmean^2 \bfunc(\roreln)\nabla\ln\romean\|_{\LinftLtwo} \right) \\
				+
{\|\rorelnt\|_{\LtwotLthree}}
\|\divbfun\|_{\LinftLtwo}\|\divbfun\|_{\LtwotLsix} 
				\\	+2\mu\|1/\romean\|_{\Linf}\|\nabla \roreln\|_{\LinftLthree}\|\divbfun\|_{\LtwotLsix}\|\nabla(\divbfun)\|_{\LtwotLtwo}
			\end{multlined}
		\end{aligned}
	\end{equation}
	for $t \in [0,\Tn]$.  By employing the assumed $r$ bound and Young's inequality, we have
	\begin{equation}
		\begin{aligned}
			\frac12\|\cmean^2 \bfunc'(\roreln)\rorelnt\|_{\LtwotLsix}\|\nabla\roreln\|_{\LtwotLthree}\|\nabla\roreln\|_{\LinftLtwo}
			\leq&\, \frac12 r \cdot\|\nabla\roreln\|_{\LtwotLthree}\|\nabla\roreln\|_{\LinftLtwo} \\
			\leq&\,  r^2 \|\roreln\|^2_{\LinftHyplusonetwo} + \frac14 C(\Omega) \|\roreln\|^2_{\LtwotHyplusonetwo}
		\end{aligned}
	\end{equation}
since $\|\cmean^2 \bfunc'(\roreln)\rorelnt\|_{\LtwonLsix} \leq r$.	Similarly,
	\begin{equation}
		\begin{aligned}
			&\|\roreln\|_{\LtwotLinf} \|\nabla\rorelnt\|_{\LtwotLtwo} \left( \|\nabla [\cmean^2 \bfunc(\roreln)]\|_{\LinftLtwo}+\|\cmean^2 \bfunc(\roreln)\nabla\ln\romean\|_{\LinftLtwo} \right) \\
			\leq&\, \begin{multlined}[t]
			\|\roreln\|_{\LtwotLinf} \|\nabla\rorelnt\|_{\LtwotLtwo} \cdot r
							\end{multlined}\\
				\leq&\,  \oneoverfoureps   C(\Omega)\|\roreln\|^2_{\LtwotHyplusonetwo}+\eps r^2 \|\rorelnt\|^2_{\LtwotHytwo}
 		\end{aligned}
	\end{equation}
for any $\eps>0$. Note that by the first embedding in \eqref{embeddings_y}, we have 
\begin{equation}
	\|\divbfun\|_{\LtwotLsix} \lesssim \|\divbfun\|_{\LtwotHone} \lesssim \|\divbfun\|_{\LtwotLtwo}+\|\nabla(\divbfun)\|_{\LtwotLtwo}.
\end{equation} Thus, we can estimate the last two terms in \eqref{rhsone I} using also the assumed $r$ bound as follows:
\begin{equation}
	\begin{aligned}
	&\begin{multlined}[t]	
{\|\rorelnt\|_{\LtwotLthree}}
\|\divbfun\|_{\LinftLtwo}\|\divbfun\|_{\LtwotLsix} \\
		\hspace*{3cm}	+2\mu \|1/\romean\|_{\Linf}\|\nabla \roreln\|_{\LinftLthree}\|\divbfun\|_{\LtwotLsix}\|\nabla(\divbfun)\|_{\LtwotLtwo} \end{multlined}\\
		\lesssim&\,\begin{multlined}[t]  r\|\divbfun\|_{\LinftLtwo}(\|\divbfun\|_{\LtwotLtwo} +\|\nabla \divbfun\|_{\LtwotLtwo})\\ \hspace*{3cm}
		+\mu\|1/\romean\|_{\Linf}r(\|\divbfun\|_{\LtwotLtwo} +\|\nabla \divbfun\|_{\LtwotLtwo})\|\nabla(\divbfun)\|_{\LtwotLtwo}.  \end{multlined}
	\end{aligned}
\end{equation}
Then by applying Young's inequality, we obtain
\begin{equation}
	\begin{aligned}
		&\begin{multlined}[t]	
			{\|\rorelnt\|_{\LtwotLthree}}
			\|\divbfun\|_{\LinftLtwo}\|\divbfun\|_{\LtwotLsix} \\
			\hspace*{3cm}	+2\mu \|1/\romean\|_{\Linf}\|\nabla \roreln\|_{\LinftLthree}\|\divbfun\|_{\LtwotLsix}\|\nabla(\divbfun)\|_{\LtwotLtwo} \end{multlined}\\
		\lesssim&\, \begin{multlined}[t] r^2 \|\divbfun\|_{\LinftLtwo}^2+ (\|\divbfun\|^2_{\LtwotLtwo} +\eps\|\nabla \divbfun\|^2_{\LtwotLtwo})+ \mu r\|\nabla \divbfun\|^2_{\LtwotLtwo}\\ \hspace*{3cm}
			+\|1/\romean\|^2_{\Linf}\|\divbfun\|^2_{\LtwotLtwo}+ \mu^2 r^2 \|\nabla \divbfun\|^2_{\LtwotLtwo} \end{multlined}	
	\end{aligned}
\end{equation}
for any $\eps>0$. Note that the presence of the $\eps$ term above will lead to a non-uniform estimate in $\mu$. In Section~\ref{sec vanishing viscosity}, we discuss a way to mend this by requiring more regularity from $\rorelnt$. By employing the derived bounds in \eqref{embeddings_y}, we arrive at
	\begin{equation}\label{est rhs1}
		\begin{aligned}
			\int_0^t \rhsone(s)\ds
			\lesssim&\, \begin{multlined}[t] 
				\frac{1}{4\eps}	\olacoef^2\|\nabla \cdot (\romean^{-1} \bff)\|^2_{\LoneLtwo} +\eps \|\divbfun \|^2_{\LinftLtwo}
				\\	+ 
			\|\roreln\|^2_{\LtwotHyplusonetwo}
				+  r^2 \|\roreln\|_{\LinftHyplusonetwo}+\eps r^2 \|\rorelnt\|_{\LtwotHytwo} 
				+r^2\|\divbfun\|_{\LinftLtwo}\\
					+(1+\|1/\romean\|^2_{\Linf})\|\divbfun\|^2_{\LtwotLtwo}+ (\mu r (1+\mu r) +\eps) \|\nabla(\divbfun)\|^2_{\LtwotLtwo} )
			\end{multlined}
		\end{aligned}
	\end{equation}
	for any $\eps>0$, where the hidden constant does not depend on $n$. The $\roreln$ and $\bfun$ terms on the right-hand side above will be either absorbed for small enough $\eps$ and $r$ or tackled via \Gronwall's inequality in the final stages of the proof. \\
	
	\noindent \underline{Estimate of $\rhstwo$}: Next, we estimate the time integral of $\rhstwo$, given in \eqref{rhstwo}, by employing H\"older's inequality as follows:
	\begin{equation}
		\begin{aligned}
			\intt\rhstwo(s)\ds
			\leq&\, \begin{multlined}[t]
				\|1/\romean\|_{\Linf} \Bigl\{	\Bigl(
				\|\nabla \hfunc(\bfun)\|_{\LtwotLtwo} \|\nabla\rorelnt\|_{\LtwotLtwo}
				\\	+\|\nabla \gprojection(\bfun)\|_{\LtwotLtwo}\|\nabla[\cmean^2\romean\bfunc(\roreln)\,\roreln]\|_{\LtwotLtwo}
				+\|\nabla \gprojection(\bfun)\|_{\LtwotLtwo}\|\nabla \hfunc(\bfun)\|_{\LtwotLtwo}\Bigr)\\
				+2\alpha_0 \Bigl(
				\tau\|(-\Delta)^{\frac{y}{4}} \rorelnt\|_{\LtwotLtwo} \|(-\Delta)^{\frac{y}{4}}\gprojection(\bfun)\|_{\LtwotLtwo}
				\\	+\eta \|(-\Delta)^{\frac{y+1}{4}} \roreln\|_{\LinftLtwo} \|(-\Delta)^{\frac{y+1}{4}}\gprojection(\bfun)\|_{\LonetLtwo}
				\Bigr) \Bigr\}.
			\end{multlined}
		\end{aligned}
	\end{equation}
	By employing Young's inequality we obtain 
	\begin{equation} \label{est rhs2}
		\begin{aligned}
			\intt\rhstwo(s)\ds
			\leq&\, \begin{multlined}[t]
				\|1/\romean\|_{\Linf} \Bigl\{	\Bigl(
				\frac{1}{4\eps} \|\nabla \hfunc(\bfun)\|^2_{\LtwotLtwo} +\eps\|\nabla\rorelnt\|^2_{\LtwotLtwo}
				\\	+\eps\|\nabla \gfunc(\bfun)\|^2_{\LtwotLtwo}+\oneoverfoureps\|\nabla[\cmean^2\romean\bfunc(\roreln)\,\roreln]\|^2_{\LtwotLtwo}
				+\eps\|\nabla \gfunc(\bfun)\|_{\LtwotLtwo}^2\\+\oneoverfoureps\|\nabla \hfunc(\bfun)\|^2_{\LtwotLtwo}\Bigr)
				+2\alpha_0 \Bigl(
				\eps \tau^2\|(-\Delta)^{\frac{y}{4}} \rorelnt\|^2_{\LtwotLtwo}+\oneoverfoureps \|(-\Delta)^{\frac{y}{4}}\gprojection(\bfun)\|^2_{\LtwotLtwo}
				\\	+\eta^2 \eps \|(-\Delta)^{\frac{y+1}{4}} \roreln\|^2_{\LinftLtwo} +\frac{1}{4\eps}\|(-\Delta)^{\frac{y+1}{4}}\gprojection(\bfun)\|^2_{\LonetLtwo}
				\Bigr) \Bigr\}
			\end{multlined}
		\end{aligned}
	\end{equation}
	for any $\eps>0$.   We further have
	\begin{equation}
	\begin{aligned}
		\eps \|\nabla \gfunc(\bfun)\|^2_{\LtwotLtwo} =& \eps \|\nabla (\bfun \cdot \nabla \ln \romean)\|^2_{\LtwotLtwo} \\
		\leq&\, 2\eps \|\nabla \bfun\|^2_{\LtwotLsix}\|\nabla \ln \romean\|^2_{\Lthree}+ 2\eps\|\bfun\|^2_{\LtwotLinf} \|\nabla(\nabla \ln \romean)\|^2_{\Ltwo}
	\end{aligned}	
	\end{equation}
	and
		\begin{equation}
		\begin{aligned}
			\oneoverfoureps \|\nabla \hfunc(\bfun)\|^2_{\LtwotLtwo} \leq&\, \begin{multlined}[t]	\frac{3}{4\eps} \|\nabla [\cmean^2 \nabla \romean] \|^2_{\Ltwo}\|\Is \bfun\|^2_{\LtwotLinf}+	\frac{3}{4\eps}\|\cmean^2 \nabla \romean\|^2_{\Lthree}\|  \Is \nabla \bfun\|^2_{\LtwotLsix}\\+\frac{3}{4\eps}\| \nabla [\cmean^2\bfdzero \cdot\nabla \romean] \|^2_{\LtwotLtwo}
				\end{multlined}
		\end{aligned}	
	\end{equation}
	and the arising $\bfun$ terms on the right-hand side can be absorbed by $\mu \intt \|\nabla(\nabla\cdot\bfun(s))\|^2_{\Ltwo} \ds$ for sufficiently small $\eps>0$ and $\deltaromeancmean$. Furthermore, 
	\begin{equation}
		\begin{aligned}
			& \oneoverfoureps\|\, \nabla[\cmean^2\romean\bfunc(\roreln)\,\roreln]\,\|^2_{\LtwotLtwo}\\
			\leq&\, \oneovertwoeps	\|\nabla[\cmean^2\romean] \bfunc(\roreln)+ \cmean^2 \romean \nabla \tfrac{B}{2A} \roreln+\cmean^2\romean\tfrac{B}{2A}\nabla \roreln\|^2_{\LinftLtwo}\|\roreln\|^2_{\LtwotLinf}+\oneovertwoeps  \|\cmean^2\romean\bfunc(\roreln)\|^2_{\LinftLinf}\|\nabla \roreln\|^2_{\LtwotLtwo}.
		\end{aligned}
	\end{equation}
	From here
		\begin{equation}
		\begin{aligned}
			& \oneoverfoureps\|\, \nabla[\cmean^2\romean\bfunc(\roreln)\,\roreln]\,\|^2_{\LtwotLtwo}\\
			\lesssim&\,\begin{multlined}[t] \left( \|\nabla[\cmean^2\romean]\|_{\Linf}^2 (1+\|\tfrac{B}{2A}\|_{\Linf}r)^2+\| \cmean^2 \romean \nabla \tfrac{B}{2A}\|_{\Lthree}^2r^2
				+\|\cmean^2 \romean \tfrac{B}{2A}\|^2_{\Linf}r^2 \right)\|\roreln\|^2_{\LtwotLinf}\\+ \|\cmean^2 \romean\|^2_{\Linf}(1+r)^2 \|\nabla \roreln\|^2_{\LtwotLtwo}
				\end{multlined}
		\end{aligned}
	\end{equation}
	and these terms can be handled via \Gronwall's inequality. \\ 
	\indent We can use the stability of $\gprojection(\bfun)$ according to Lemma~\ref{lemGprojg} to estimate 
	\begin{equation}
		\begin{aligned}
			&\oneoverfoureps \|(-\Delta)^{\frac{y}{4}}\gprojection(\bfun)\|^2_{\LtwotLtwo}+ \frac{1}{4\eps}\|(-\Delta)^{\frac{y+1}{4}}\gprojection(\bfun)\|^2_{\LonetLtwo} \\
			\leq&\, C \|\nabla\ln\romean\|_{\Linf\cap\Hyhalf}\|\bfun\|_{L^2_t(\Linf\cap\Hyhalf)}+C \sqrt{T}\|\nabla\ln\romean\|_{\Linf\cap\Hyplusonetwo}\|\bfun\|_{L^2_t(\Linf\cap\Hyplusonetwo)},
		\end{aligned}
	\end{equation}
	where $C$ does not depend on $n$, and absorb these terms for sufficiently small $\|\nabla\ln\romean\|_{\Linf\cap H^{\frac{y+1}{2}}(\Omega)}$ (that is, $\delta_{\romean, \cmean}$). \\
	
	\noindent \underline{Combining the bounds}:	By employing \eqref{est rhs1} and \eqref{est rhs2} in the time-integrated identity \eqref{energy identity}, taking the supremum over $t\in(0,\tau)$ for $\tau \in (0, \Tn)$ and reducing $\eps$ and $r$ (independently of $n$), we end up with
	\begin{equation}\label{enest}
		\begin{aligned}
			&\esssup_{t \in (0,\tau)}\calE(t) +  \int_0^\tau  \left(\mu \|\nabla(\nabla\cdot\bfun(s))\|^2_{\Ltwo} + \|(-\Delta_N)^{\frac{y}{4}} \, \rorelnt(s)\|^2_{\Ltwo} \right)\ds \\
			\leq&\, \begin{multlined}[t]C \Bigl(\|\bfuzero\|^2_{\Hdiv} 
				+ \|\cmean\sqrt{\rorelzero}\, \nabla \rorelzero\|^2_{\Ltwo}
				+\|\rorelzero \|^2_{\Hyplusonetwo}+\frac{1}{4\eps}	\|\nabla \cdot (\romean^{-1} \bff)\|^2_{\LoneLtwo}\\+\|\roreln\|^2_{L^2(0,\tau; \Hyplusonetwo)}+(1+\|1/\romean\|^2_{\Linf})\|\divbfun\|^2_{L^2(0, \tau; \Ltwo)} +\Tn\| \nabla [\cmean^2\bfdzero \cdot \nabla \romean] \|^2_{\Ltwo}\Bigr) \end{multlined}
		\end{aligned}
	\end{equation}
	for $\tau \in (0, \Tn)$, where $C$ does not depend on $\Tn$ or $n$.  Above, we have also used the boundedness of approximate initial data:
	\begin{equation}
		\begin{aligned}
			\|\nabla\cdot\bfunzero\|_{\Ltwo} \lesssim\|\bfuzero\|_{\Hdiv}, \quad \|\rorelnzero\|_{\Hyplusonetwo} \lesssim \|\rorelzero\|_{\Hyplusonetwo}.
		\end{aligned}
	\end{equation}
\indent Estimate \eqref{enest} does not contain a bound on $\pn$, which we obtain in the final step of the proof. 
	To this end, we test the time-integrated version of $\moG$\ with $\nabla(\It \pn) \in \Ltwod$ for $t \in [0, \Tn]$, which yields, after integration over $\Omega$,
	\begin{equation}
		\begin{aligned}
			\begin{multlined}[t] \int_\Omega \Bigl\{\bigl(\romean(\bfun-\bfunzero)+\nabla \It \pn-\mu\nabla(\nabla\cdot \It \bfu)-\It\bff\bigr)\cdot \nabla(\It \pn)\dx =0.
			\end{multlined}
		\end{aligned}
	\end{equation}
	From here we obtain
	\begin{equation}\label{est_Itp}
		\begin{aligned}
			\|\nabla \It \pn\|_{L^2(0, \tau; \Ltwo)}\leq&\, \begin{multlined}[t] \|\It \bff+\mu\nabla(\nabla\cdot \It \bfun)-\romean(\bfun-\bfunzero)\|_{L^2(0, \tau; \Ltwo)} 
					\end{multlined}\\
					\lesssim&\, \|\It \bff\|_{\LtwoLtwo}	+ \mu \|\nabla(\nabla\cdot \It \bfun)\|_{L^2(0, \tau; \Ltwo)}+\|\romean\|_{\Linf}\|\bfun\|_{L^2(0, \tau; \Ltwo)}+ \|\romean\|_{\Linf}\|\bfunzero\|_{\Ltwo}.
		\end{aligned}
	\end{equation}
	Note that we cannot obtain a bound on $\nabla \pn$ from $\moG$\ because we lack a bound on $\bfunt$. Squaring this estimate, multiplying it by $\lambda>0$ and adding it to \eqref{enest} with $\lambda$ sufficiently small (independently of $n$) leads to
		\begin{equation}\label{enest}
		\begin{aligned}
			&\esssup_{t \in (0,\tau)}\calE(t) +  \esssup_{t \in (0,\tau)}\calD(t) \\
			\leq&\, \begin{multlined}[t]C \Bigl((1+\|\romean\|^2_{\Linf})	 \|\bfuzero\|^2_{\Hdiv} 
				+ \|\cmean\sqrt{\rorelzero}\, \nabla \rorelzero\|^2_{\Ltwo}
				+\|\rorelzero \|^2_{\Hyplusonetwo}+	\|\nabla \cdot (\romean^{-1} \bff)\|^2_{\LoneLtwo}+\|\It \bff\|^2_{\LtwoLtwo} \\+\|\romean\|^2_{\Linf}\|\bfun\|^2_{L^2(0, \tau; \Ltwo)}+\|\roreln\|^2_{L^2(0,\tau; \Hyplusonetwo)}+(1+\|1/\romean\|_{\Linf})\|\divbfun\|^2_{L^2(0, \tau; \Ltwo)} \\+\Tn\| \nabla [\cmean^2\bfdzero \cdot \nabla \romean] \|^2_{\Ltwo}\Bigr) \end{multlined}
		\end{aligned}
	\end{equation}
	for $t \in [0, \Tn]$, where the constant $C$ does not depend on $n$. 
	By employing \Gronwall's inequality, we arrive at the claimed estimate.
\end{proof}

\subsection[\texorpdfstring{Extending the existence interval to [0,T]}{Extending the existence interval to [0,T]}]{Extending the existence interval to $[0,T]$}
Equipped with a uniform bound in \eqref{uniform energy bound Galerkin}, we can now extend the existence interval of Galerkin approximations to $[0,T]$. We do so by proving that for small enough data, the uniform boundedness assumption made in \eqref{smallness condition acoef bcoef} holds.
\begin{proposition}
	Let the assumptions of Lemma~\ref{lemma bfun} and Proposition~\ref{prop roreln pn} hold. Let assumption
\eqref{assumptions y} on $y$
as well as assumption
\eqref{smallness condition romean cmean} on the smallness of gradients of $\romean$ and $\cmean^2$ hold with the bound $\deltaromeancmean$. Further, let $(\bfun, \roreln, \pn)$ be the solution of \eqref{Galerkin approximate system} on $[0, \Tn]$.
Then there exists $\delta>0$, independent of $n$, such that if
	\begin{equation} \label{smallness condition data}
		\|\rorelzero\|^2_{\Hyplusonetwo} +\|\bfuzero\|^2_{\Hdiv} +\|\bfdzero\|^2_{\Linf\cap\Hone}+ \|\bff\|^2_{\Xf} \leq \delta,
	\end{equation}
and $\deltaromeancmean$ is small enough, independent of $n$, then the following uniform bound holds:
	\begin{equation}
		\begin{aligned}
			\begin{multlined}[t]
				\calL(\roreln, \bfun)(t) :=	\|\roreln(t)\|_{\Linf}+	2\|\rorelnt\|_{L^2(0,t; \Lthree)}+2\|\nabla \roreln(t)\|_{\Lthree} + \frac12 \|(B/A) c_0^2  \rorelnt\|_{L^2(0,t; \Lsix)}\\
				+\|\nabla [\cmean^2 \bfunc(\roreln)](t)\|_{\Ltwo}+\|\cmean^2 \bfunc(\roreln)(t)\nabla\ln\romean\|_{\Ltwo}< r
			\end{multlined}
		\end{aligned}
	\end{equation}
	for all $t \in [0,\Tn]$. 
	Consequently, $\Tn=T$ can be chosen independent of $n$.
\end{proposition}
\begin{proof}
	We argue by contradiction. Assume that there exists $t_0 \in [0, \Tn]$, such that
	\begin{equation}
		\calL(\roreln, \bfun)(t_0)>r.
	\end{equation}
	Let $t_*= \inf \{t: 	\calL(\roreln, \bfun)(t)>r\}$. By continuity of $\calL$, then
	\begin{equation}
		\calL(\roreln, \bfun)(t_*) = r.
	\end{equation} 
	However, since $\calL(\roreln, \bfun)(t_*) = r$, we know from the energy bound \eqref{uniform energy bound Galerkin} that
	\begin{equation} \label{est energy t star}
		\begin{aligned}
			\calE(t_*)+\calD(t_*) \leq C \delta.
		\end{aligned}
	\end{equation}
	Furthermore, by employing the Sobolev embeddings in \eqref{embeddings_y}, it follows that
	\begin{equation} \label{est L}
		\calL^2(\roreln, \bfun)(t_0) \leq C_0 (\calE(t_*)+\calD(t_*)),
	\end{equation}
	where (crucially) the constant $C_0$ does not depend on $\Tn$ or $n$. Combining estimates \eqref{est energy t star} and \eqref{est L} yields
	\begin{equation}
		\calL^2(\roreln, \bfun)(t_*) \leq C C_0 \delta.
	\end{equation}
	Choosing the size of data to be $\delta< \frac{r^2}{C C_0}$ leads to $ \calL(\roreln, \bfun)(t_*) < r$ and thus a contradiction. \\
	\indent By Proposition~\ref{prop: Galerkin energy estimate}, the uniform boundedness of $\calL$ in turn implies that the energy is uniformly bounded:
	\[
	\calE(t) \leq C, \quad  t \in [0, \Tn],
	\]
	and we can thus prolong Galerkin solutions until we reach the final time $T$.
\end{proof}		

\subsection[\texorpdfstring{Passing to the limit as $\boldsymbol{n \rightarrow \infty}$}{Passing to the limit as n -> inf}]{Passing to the limit as $\boldsymbol{n \rightarrow \infty}$}
Thanks to the established $n$-uniform bounds on Galerkin approximations on $[0,T]$, we may extract subsequences of $\{\bfun\}_{n\geq 1}$ and $\{\roreln\}_{n\geq 1}$, which we do not relabel, such that
\begin{equation} \label{weak limits bfu}
	\begin{alignedat}{4} 
		\bfun  &\relbar\joinrel\rightharpoonup \bfu &&\text{ weakly-$*$}  &&\text{ in } && \LinfTHdiv,  \\
		\divbfun  &\relbar\joinrel\rightharpoonup \divbfu  &&\text{ weakly-$*$}  &&\text{ in } && \LinfTLtwo,  \\
		\nabla (\divbfun)  &\relbar\joinrel\rightharpoonup \nabla (\divbfu)\quad &&\text{ weakly}  &&\text{ in } &&\LtwoTLtwo,  
	\end{alignedat} 
\end{equation}
and
\begin{equation} \label{weak limits roreln pn}
	\begin{alignedat}{4} 
		\roreln  &\relbar\joinrel\rightharpoonup \rorel &&\text{ weakly-$*$}  &&\text{ in } &&L^\infty(0,T; H^{\frac{y+1}{2}}(\Omega)),  \\
		\rorelnt  &\relbar\joinrel\rightharpoonup \rorelt &&\text{ weakly}  &&\text{ in } &&L^2(0,T; H^{\frac{y}{2}}(\Omega)).
	\end{alignedat} 
\end{equation}
By the compact embedding $\Xsigma \hookrightarrow \hookrightarrow C([0,T]; \Hone)$, we also know that there is a subsequence of $\{\roreln\}_{n \geq 1}$, not relabeled, such that
\begin{equation} \label{strong limits roreln}
	\begin{aligned} 
		\roreln  \rightarrow \rorel \quad \text{strongly in } \, C([0,T]; \Hone).
	\end{aligned} 
\end{equation}
Additionally, by the uniform boundedness of $\It \pn$, there is a subsequence of $\{\It \pn)\}_{n\geq 1}$, again not relabeled, such that
\begin{equation} \label{weak limits pn}
	\begin{aligned}
		\It \pn \relbar\joinrel\rightharpoonup \It p \quad \text{ weakly} \ \text{ in } \ \LtwoTHone.
	\end{aligned} 
\end{equation}
Thanks to \eqref{weak limits bfu} and \eqref{weak limits pn}, we can immediately pass to the limit as $n \rightarrow \infty$ in 
\begin{equation}
	\begin{aligned}
		\intTO \left(\romean(\bfun-\bfunzero)+\nabla\It \pn - \mu \nabla (\nabla \cdot \It \bfun) - \It \bff\right) \cdot \bfv \dxt =0, \quad \bfv \in \LtwoTLtwod
	\end{aligned}
\end{equation}
to conclude that 
\begin{equation}
	\begin{aligned}
		\intTO \left(\romean(\bfu-\bfuzero)+\nabla\It p - \mu \nabla (\nabla \cdot \It \bfu) - \It \bff\right) \cdot \bfv \dxt =0 \quad \text{for all }\ \bfv \in \LtwoTLtwod.
	\end{aligned}
\end{equation}
Next, to pass to the limit in $\maG$, we first note that for any $w \in \LtwoTLtwo$ we have 
\begin{equation} \label{weak conv nolinear terms}
	\begin{alignedat}{4} 
		\intTO	(\roreln \divbfun  - \rorel \divbfu) w \dxt = \intTO	(\roreln -\rorel) \divbfun  w \dxt+\intTO	(\divbfun  -  \divbfu) \rorel w \dxt.
	\end{alignedat} 
\end{equation}
The convergence of the second term to zero as $n \rightarrow \infty$ is immediate since $\rorel w \in \LtwoTLtwo$ for $w \in \LtwoTLtwo$. The convergence of the first term to zero follows by
\begin{equation}
	\begin{aligned}
		\intTO	(\roreln -\rorel) \divbfun  w \dxt \leq C\|\roreln-\rorel\|_{C([0,T]; \Hone)}\|\divbfun\|_{L^2(0,T; \Lfour)} \|w\|_{\LtwoTLtwo}
	\end{aligned}
\end{equation}
and \eqref{strong limits roreln}.
Next, we fix $N$ and choose 
\begin{equation} \label{functions v}
	v(t) = \displaystyle \sum_{i=1}^N \xisigma_i (t)w_i(x), \quad 	\phi(t) = \displaystyle \sum_{i=1}^N \xip_i (t)w_i(x)
\end{equation}
where $\{\xisigma\}_{i=1}^N$ and $\{\xip\}_{i=1}^N$ are given smooth functions. We choose $n \geq N$ and note that $\roreln$ satisfies
\begin{equation}
	\begin{aligned}
		\intTO \left( \rorelnt + \afunc(\roreln) \divbfun - \gfunc(\bfun)\right) v \dxt =0.
	\end{aligned}
\end{equation}
Thanks to the convergence of \eqref{weak conv nolinear terms} to zero as $n \rightarrow \infty$, we can then pass to the limit as $n \rightarrow \infty$ in the above equation and use the density of functions of the form \eqref{functions v} in $\LtwoTLtwo$ to conclude that
\begin{equation}
	\begin{aligned}
		\intTO \left( \rorelt + \afunc(\rorel) \divbfu - \gfunc(\bfu)\right) v \dxt =0 \quad \text{for any}\ v\in \LtwoTLtwo.
	\end{aligned}
\end{equation}
Similarly to the arguments in, e.g.,~\cite[Ch.\ 7]{evans_1998}, with $v(T)=0$, we have
\begin{equation}
	-	\intO \roreln(0) v(0) \dx-\intTO \rorel v_t \dxt+\intTO \left(\afunc(\roreln) \divbfun - \gfunc(\bfun)\right) v \dxt =0.
\end{equation}
By passing to the limit as $n \rightarrow \infty$ and using the analogous identity for $\rorel$, we can show that $\rorel(0)=\rorelzero$ since $\roreln(0) \rightarrow \rorelzero$ in $\Ltwo$. With similar reasoning to \eqref{weak conv nolinear terms}, 
\begin{equation} 
	\begin{aligned} 
		\intTO ( \It(\bfunc(\roreln)\roreln) -\It(\bfunc(\rorel)\rorel)) \phi\dxt = \intTO  \It(\roreln-\rorel) \phi \dxt+\frac{B}{2A}\intTO  \It((\roreln-\rorel)(\roreln+\rorel)) \phi \dxt \rightarrow 0
	\end{aligned} 
\end{equation}
as $n \rightarrow \infty$, thanks to \eqref{weak limits roreln pn} and \eqref{strong limits roreln}.
We can then pass to the limit also in 
\begin{equation}
	\begin{aligned}
	\begin{multlined}[t]	\intTO \left\{(\It \pn - \cmean^2 \romean \It (\bfunc(\roreln) \roreln)- \It \hfunc(\bfun)) \Deltaromean\phi 
		+ 2\alpha_0 \bigl(\tau(-\Delta)^{\frac{y}{4}} (\roreln-\rorelnzero) (-\Delta)^{\frac{y}{4}} \phi
		+\eta (-\Delta)^{\frac{y+1}{4}} \It \roreln (-\Delta)^{\frac{y+1}{4}} \phi \bigr) \right\}  \dx\\ =0, 
		\end{multlined}
	\end{aligned}
\end{equation}
using that
\begin{equation}
	\begin{aligned}
		((-\Delta)^{\frac{y}{4}} (\rorelzero -\rorelnzero), (-\Delta)^{\frac{y}{4}} \phi)_{\Ltwo}  \rightarrow 0 \quad \text{as } n \rightarrow \infty.
	\end{aligned}
\end{equation}
Altogether, we conclude that $(\bfu, \rorel, p)$ is a solution of the problem in the sense of Definition~\ref{def solution}.
~\\

By passing to the limit in the semi-discrete energy estimate \eqref{uniform energy bound Galerkin} and utilizing the lower semi-continuity of norms, we find that $(\bfu, \rorel, p)$ satisfies an  energy bound analogous to \eqref{uniform energy bound Galerkin} and arrive at the following existence result.

\begin{theorem}[Existence of solutions] \label{theorem wellposedness regularized problem} Let $T>0$. Let $\mu>0$ and $d-1 <y$, $2 \leq y \leq 3$ (cf. \eqref{assumptions y}) and assume that
	\begin{equation}
		\bfuzero \in \Hdiv, \quad \bfdzero \in \Linf \cap \Hone, \quad \rorelzero \in \Hyplusonetwo, \quad \bff \in \Xf,
	\end{equation}
	and	$B/A \in \XBonA$, $\romean \in \Xromean$, $\cmean^2 \in \Xcmean$, where the spaces $\Xf$, $\XBonA$, $\Xromean$, and $\Xcmean$ are defined in \eqref{Xf}, \eqref{X BA}, \eqref{Xromean}, and \eqref{Xcmean} respectively. There exist $\delta>0$ and $\deltaromeancmean>0$, such that if the smallness conditions \eqref{smallness condition romean cmean} and
	\eqref{smallness condition data} 
	hold, then there exists a solution $(\bfu, \rorel, p)$ of \eqref{regularized balance-state} in the sense of Definition~\ref{def solution}, which satisfies the following bound:
	\begin{equation}\label{nonuniform energy estimate}
	\begin{aligned}
		&\begin{multlined}[t]
			\|\bfu\|^2_{\LinfHdiv} 
			+ \|\nabla \rorel\|^2_{\LinfLtwo} + \|\rorel\|^2_{\LinfHyplusonetwo} 
			+ \intT  \left(\mu \|\nabla(\nabla\cdot\bfu(t))\|^2_{\Ltwo} + \|\rorelt(t)\|^2_{\Hytwo}+ \|\nabla \It p\|^2_{\Ltwo} \right)\dt
			\end{multlined}\\
		\leq&\, \begin{multlined}[t] C_1 \exp(C_2T) \Bigl(\|\nabla \cdot (\romean^{-1} \bff)\|^2_{\LtwoLtwo}+\|\It \bff\|^2_{\LtwoLtwo} 
			+ \|\bfuzero\|^2_{\Hdiv} 
			+ \|\cmean\sqrt{\rorelzero}\, \nabla \rorelzero\|^2_{\Ltwo}
			+\|\rorelzero \|^2_{\Hyplusonetwo}\\ +\| \nabla [\cmean^2\bfdzero \cdot \nabla \romean] \|^2_{\Ltwo}	\Bigr). \end{multlined}
	\end{aligned}
\end{equation}
\end{theorem}	
The estimate in \eqref{nonuniform energy estimate} is not uniform in $\mu$; that is, we cannot use this result to investigate the limit of solutions as $\mu \searrow 0$. As the setting $\mu=0$ is of interest for working with \eqref{original system}, we investigate it next by modifying the assumptions on $y$ as well as the functions $\gfunc$ and $\hfunc$.
		\section{The vanishing viscosity limit under stronger assumptions} \label{sec vanishing viscosity}
	
	In this section, we discuss the vanishing $\mu$ limit of solutions to the problem with $\gfunc=\hfunc \equiv 0$:
		\begin{equation} \label{weak form gcoef hcoef zero}
		\begin{aligned}
			\begin{multlined}[t]		\intTO \Bigl\{\left(\romean(\bfu-\bfuzero)+\nabla\It p - \mu \nabla (\divbfu) - \It \bff\right) \cdot \bfv 
				+\left( \rorelt + \afunc(\rorel) \divbfu \right) v 
				\\+	\left(\It p - \cmean^2 \romean \It (\bfunc(\rorel) \rorel) \right) \Deltaromean\phi+2\alpha_0 \bigl(\tau(-\Delta)^{\frac{y}{4}} (\rorel-\rorelzero) (-\Delta)^{\frac{y}{4}} \phi
				+\eta (-\Delta)^{\frac{y+1}{4}} \It \rorel (-\Delta)^{\frac{y+1}{4}} \phi \bigr) \Bigr\}  \dxt =0 \end{multlined}
		\end{aligned}
	\end{equation}
which holds	for all $\bfv \in \LtwoTLtwod$, $v \in \LtwoTLtwo$, and $\phi \in L^2(0,T; \Hyplusonetwo)$, such that $\nabla \phi \cdot \nu =0$. \\
\indent Looking at the energy estimates in the previous section starting from the identity in \eqref{energy identity}, we see that we can simplify them because now $\rhstwo \equiv 0$. 	Since $\gfunc=\hfunc=0$, we can assume slightly less regularity of the coefficients as compared to \eqref{Xromean}, \eqref{Xcmean}, namely that 
\begin{equation} \label{weaker assumptions romean cmean}
	\romean,\romeaninv,\cmean, \frac{1}{\cmean}, B/A \in \Linf,\quad \romean\in\Htwo,\quad \cmean^2, B/A\in W^{1,3}(\Omega).
\end{equation} 
Furthermore, there is no need for the initial condition on $\bfd$. By then re-examining the derivation of the estimate of the time-integrated $\rhsone$, we observe that the culprit in \eqref{rhsone} for the non-uniform bounds in $\mu$ was the term
\begin{equation} \label{est rorelt}
	\intt \rorelnt |\divbfun| \ds \leq \intt \|\rorelnt\|_{\Lsix}\|\divbfun\|_{\Ltwo}\|\divbfun\|_{\Lthree}\ds,
\end{equation}
in particular, the need to further bound $\|\divbfun\|_{\Lthree}$; see \eqref{rhsone I}. We can obtain a $\mu$-uniform energy estimate if we have the following bound:
	\begin{equation} \label{embed_Linf_Hytwo}
		\| \rorelnt \|_{\Linf} \leq C \|\rorelnt \|_{\Hytwo} .
	\end{equation}
	because we can then replace estimate \eqref{est rorelt} by
	\begin{equation}
		\intt \rorelnt |\divbfun| \ds \leq \intt\|\rorelnt\|_{\Linf}\|\divbfun\|^2_{\Ltwo}\ds.
		\end{equation}
	For this reason, here we strengthen the lower bound on $y$ to $y>d$, so that embedding estimate \eqref{embed_Linf_Hytwo} holds. (Note that since $\gfunc=\hfunc=0$, we do not need the condition $y \leq 3$ any longer). By otherwise proceeding as in the previous section via the Faedo--Galerkin procedure, we arrive at the following uniform in $\mu$ result.
	\begin{proposition} \label{proposition uniform wellposedness  mu} Let $T>0$. Let $\mu>0$ and $y>d$ and assume that
		\begin{equation}
			\bfuzero \in \Hdiv, \quad \rorelzero \in \Hyplusonetwo, \quad \bff \in \Xf,
		\end{equation}
		and that \eqref{weaker assumptions romean cmean} holds. There exists $\delta>0$, such that if 
		\begin{equation}
			\|\rorelzero\|^2_{\Hyplusonetwo} +\|\bfuzero\|^2_{\Hdiv} + \|\bff\|^2_{\Xf} \leq \delta,
		\end{equation}
		then there exists a solution $(\bfu, \rorel, p)$ of \eqref{weak form gcoef hcoef zero}, which satisfies the following bound:
		\begin{equation}\label{uniform energy bound}
			\begin{aligned}
				&\begin{multlined}[t]
					\|\bfu\|^2_{\LinfHdiv} 
					+ \|\nabla \rorel\|^2_{\LinfLtwo} + \|\rorel\|^2_{\LinfHyplusonetwo} 
					+ \intT  \left(\mu \|\nabla(\nabla\cdot\bfu(t))\|^2_{\Ltwo} + \|\rorelt(t)\|^2_{\Hytwo}+ \|\nabla \It p\|^2_{\Ltwo} \right)\dt
				\end{multlined}\\
				\leq&\, \begin{multlined}[t] C_1 \exp(C_2T) \Bigl(\|\nabla \cdot (\romean^{-1} \bff)\|^2_{\LtwoLtwo}+\|\It \bff\|^2_{\LtwoLtwo} 
					+ \|\bfuzero\|^2_{\Hdiv} 
					+ \|\cmean\sqrt{\rorelzero}\, \nabla \rorelzero\|^2_{\Ltwo}
					+\|\rorelzero \|^2_{\Hyplusonetwo}	\Bigr),\end{multlined}
			\end{aligned}
		\end{equation}
		where the constants $C_1$ and $C_2$ do not depend on $\mu$.
	\end{proposition}	
	As \eqref{uniform energy bound} provides a $\mu$-uniform bound on $(\bfu, \rorel, \It p)$, similarly to Section~\ref{Sec:Analysis},  we can find subsequences of $\{\bfu\}_{\mu>0}$ and $\{\rorel\}_{\mu>0}$, which we do not relabel, such that
	\begin{equation} \label{weak limits bfu mu}
		\begin{alignedat}{4} 
			\bfu  &\relbar\joinrel\rightharpoonup \bfumuzero \quad&&\text{ weakly-$*$}  &&\text{ in } && \LinfTHdiv,  \\
			\divbfu  &\relbar\joinrel\rightharpoonup \divbfumuzero  \quad&&\text{ weakly-$*$}  &&\text{ in } && \LinfTLtwo,  
		\end{alignedat} 
	\end{equation}
	and
	\begin{equation} \label{weak limits roreln pn mu}
		\begin{alignedat}{4} 
			\rorel  &\relbar\joinrel\rightharpoonup \rorelmuzero \quad &&\text{ weakly-$*$}  &&\text{ in } &&L^\infty(0,T; H^{\frac{y+1}{2}}(\Omega)),  \\
			\rorelt  &\relbar\joinrel\rightharpoonup \rorelmuzero_t \quad &&\text{ weakly}  &&\text{ in } &&L^2(0,T; H^{\frac{y}{2}}(\Omega))
		\end{alignedat} 
	\end{equation}
as $\mu \searrow 0$.	
	Additionally, 
	\begin{equation} \label{weak limits pn mu}
		\begin{aligned}
			\It p \relbar\joinrel\rightharpoonup \It \pmuzero \quad \text{ weakly} \ \text{ in } \ \LtwoTHone \ \text{as } \mu \searrow 0.
		\end{aligned} 
	\end{equation}
	Furthermore, from \eqref{uniform energy bound} we have the uniform bound
	\begin{equation}
		\sqrt{\mu} \|\nabla (\divbfu)\|_{\LtwoLtwo} \leq C
	\end{equation}
	and thus know that
	\begin{equation}
		\intTO \mu \nabla (\divbfu) \cdot \bfv \dxt \rightarrow 0 \quad \text{ as } \mu \searrow 0.
	\end{equation}
	These convergence results allow us to pass to the limit in \eqref{weak form gcoef hcoef zero} and prove existence of solutions for the problem without viscosity. To give the statement, we introduce the space
	\begin{equation} \label{def Xu}
		\Xu = \left\{\bfu\in \LinfTHdiv:  \ \bfu \cdot \nu=0 \text{ on }\partial \Omega \right\}.
	\end{equation}
Recall that the spaces $\Xrho$ and $\XItp$ are defined in \eqref{def Xromean} and \eqref{def XItp}, respectively.	We can thus state the second main result of this work.
		\begin{theorem}[Existence of solutions when $\mu=0$] \label{theorem wellposedness  mu ero} Under the assumptions of Proposition~\ref{proposition uniform wellposedness  mu}, there exists $(\bfumuzero, \rorelmuzero, \pmuzero) \in \calX= \Xu \times \Xsigma \times \XItp$ that satisfies
					\begin{equation} \label{weak form mu zero}
				\begin{aligned}
					\begin{multlined}[t]		\intTO \Bigl\{\left(\romean(\bfumuzero-\bfuzero)+\nabla\It \pmuzero - \It \bff\right) \cdot \bfv 
						+\left( \rorelmuzero_t + \afunc(\rorelmuzero) \divbfumuzero \right) w 
						\\+	\left(\It \pmuzero - \cmean^2 \romean \It (\bfunc(\rorelmuzero) \rorelmuzero) \right) \Deltaromean\phi+2\alpha_0 \bigl(\tau(-\Delta)^{\frac{y}{4}} (\rorelmuzero-\rorelzero) (-\Delta)^{\frac{y}{4}} \phi
						+\eta (-\Delta)^{\frac{y+1}{4}} \It \rorelmuzero (-\Delta)^{\frac{y+1}{4}} \phi \bigr) \Bigr\}  \dxt\\ =0 \end{multlined}
				\end{aligned}
			\end{equation}
			for all $\bfv \in \LtwoTLtwod$, $w \in \LtwoTLtwo$, and $\phi \in L^1(0,T; \Hyplusonetwo)$, such that $\nabla \phi \cdot \nu =0$ with $\rorelmuzero\vert_{t=0}=\rorelzero$. Furthermore, the following bound holds:
		\begin{equation}\label{uniform energy bound}
			\begin{aligned}
				&\begin{multlined}[t]
					\|\bfumuzero\|^2_{\LinfHdiv} 
					+ \|\nabla \rorelmuzero\|^2_{\LinfLtwo} + \|\rorelmuzero\|^2_{\LinfHyplusonetwo} 
					+ \intT  \left(\|\rorelmuzero_t(t)\|^2_{\Hytwo}+ \|\nabla \It \pmuzero\|^2_{\Ltwo} \right)\dt
				\end{multlined}\\
				\leq&\, \begin{multlined}[t] C_1 \exp(C_2T) \Bigl(\|\nabla \cdot (\romean^{-1} \bff)\|^2_{\LtwoLtwo}+\|\It \bff\|^2_{\LtwoLtwo} 
					+ \|\bfuzero\|^2_{\Hdiv} 
					+ \|\cmean\sqrt{\rorelzero}\, \nabla \rorelzero\|^2_{\Ltwo}
					+\|\rorelzero \|^2_{\Hyplusonetwo}	\Bigr),\end{multlined}
			\end{aligned}
		\end{equation}
		where the constants $C_1$ and $C_2$ do not depend on $\mu$.
	\end{theorem}	
	With this result, we have established sufficient conditions for the existence of solutions to \eqref{original system} with the modified absorption operator \eqref{L_almostoriginal}, where the problem is understood in the sense of \eqref{weak form mu zero}. 	As previously mentioned, the theory can also be adapted to allow for having the original absorption operator \eqref{L_JRT16},
	however at the cost of higher smoothness of the coefficients $\romean$ and $\cmean$.

\section*{Acknowledgements}
The contribution of B.K.\ to this work was supported by the Austrian Science Fund (FWF) [10.55776/P36318]. B.C.\ acknowledges the support of the Engineering and Physical Sciences Research Council, UK [EP/W029324/1, EP/T014369/1].

\section*{Appendix} \label{appendix}

We here provide the proof of Lemma~\ref{lemGprojg}, which is partly based on the following consequence of the Courant--Fischer $\max-\min$ formula of eigenvalues of compact selfadjoint operators, adapted from~\cite{aaoBayes}.
\begin{lemma}\label{lem:eigvaltransf}
Let $V$ and $H$ be Hilbert spaces with $C:V \to V$ self-adjoint and compact and $\mathcal{M}:V \to H$ boundedly invertible with $\mathcal{M}\in L(V,H)$ and $\mathcal{M}^{-1} \in L(H,V)$. Then the operator $\tilde{C}:= (\mathcal{M}^{-1})^*C \mathcal{M}^{-1} : H \to H$ is self-adjoint and compact and the eigenvalues $\lambda_k$ of $C$ and $\mu_k$ of $\tilde{C}$ decay at the same rate; more precisely, it holds
\begin{equation*}
\frac{1}{\|\mathcal{M}^{-1}\|^2}\mu_k \leq \lambda_k \leq \|\mathcal{M}\|^2 \mu_k,
\end{equation*}
with $\lambda_1 \geq \lambda_2 \geq \cdots \geq 0$ and $\mu_1 \geq \mu_2 \geq \cdots \geq 0$.
\end{lemma}
\begin{proof} 
Recall that by the Courant--Fischer Theorem, the eigenvalues in decreasing order obey the following variational characterization:
\begin{equation*}
\lambda_k = \max \{ \min\{(Cx,x)_V: x\in S_k, \|x\|=1\}: \dim (S_k)=k, S_k \mbox{ subspace of V}\}. 
\end{equation*}
From this characterization, we obtain
\begin{align*}
\lambda_k 
=&\, \max_{\dim(S_k)=k} \min_{x \in S_k, \|x\| = 1} (Cx,x)_V \\
=&\,  \max_{\dim(S_k)=k} \min_{x \in S_k, \|x\|=1} ((\mathcal{M}^{-1})^*C\mathcal{M}^{-1}\, \mathcal{M}x, \mathcal{M}x)_V  \\
=&\,\max_{\dim(S_k)=k} \min_{x \in S_k, \hat{x} = \mathcal{M}x / \|\mathcal{M}x\|, \|x\|=1} (\tilde{C}\hat{x}, \hat{x})_H \|\mathcal{M}x\|^2 \\
\geq&\,\frac{1}{\|\mathcal{M}^{-1}\|}
\max_{\dim(S_k)=k} \min_{x \in S_k, \hat{x} = \mathcal{M}x / \|\mathcal{M}x\|, \|x\|=1} (\tilde{C}\hat{x}, \hat{x})_H=
\ (\star),
\end{align*}
using $\|\mathcal{M}x\|\geq\frac{1}{\|\mathcal{M}^{-1}\|}\|x\|=\frac{1}{\|\mathcal{M}^{-1}\|}$.
Due to the fact that
\begin{equation*}
\hat{x} = \frac{\mathcal{M}x}{\|\mathcal{M}x\|} \in \hat{S}_k = \mathcal{M}S_k,
\end{equation*}
and the dimension of $S_k$ being $k$, due to regularity of $\mathcal{M}$, $\hat{S}_k$ is of dimension $k$ as well. Therefore, taking the minimum over a superset by dropping the constraint $\|x\|=1$ results in
\begin{equation*}
 (\star) \geq \frac{1}{\|\mathcal{M}^{-1}\|^2} \max_{\dim(\hat{S}_k)=k} \min_{ \hat{x} \in \hat{S}_k, \|\hat{x}\| =1}(\tilde{C}\hat{x},\hat{x})_H = \frac{1}{\|\mathcal{M}^{-1}\|^2} \mu_k.
\end{equation*}
Analogously, it holds
\begin{equation*}
\mu_k  \geq \frac{1}{\|(\mathcal{M}^{-1})^{-1}\|^2} \lambda_k = \frac{1}{\|\mathcal{M}\|^2} \lambda_k,
\end{equation*}
which concludes the proof.
\end{proof}    

Proof of Lemma ~\ref{lemGprojg}.
\begin{proof}
The first estimate in \eqref{stability bounds grpojection} follows by testing \eqref{GalProj} with $v^n=\gcoefprojection$ and using the Cauchy--Schwarz inequality as well as the estimate 
\[
\|\nabla \gfunc(\bfun)\|_{\Ltwo}\leq \|\nabla\nabla\ln\romean\|_{\Lthree}\|\bfun\|_{\Lsix} + \|\nabla\ln\romean\|_{\Linf}\|\nabla\bfun\|_{\Ltwo},
\]
where $\nabla\nabla$ denotes the Hessian. For the second and third bounds in \eqref{stability bounds grpojection}, we recall the definition of the fractional power of a symmetric nonnegative  operator $A$ with eigensystem $\{(\lambda_i,\wi)\}_{i \geq 1}$ as 
\begin{equation} \label{def Agamma}
A^\gamma v=\sum_{i\in\mathbb{N}}\lambda_i^\gamma (v,\wi).
\end{equation}
We also note that
\[
\begin{aligned}
	\|(-\Delta_N)^\gamma\gcoefprojection\|_{\Ltwo}
	\leq \|\romean\|_{\Linf}^\gamma\|(-\Deltaromean)^\gamma\gcoefprojection\|_{\Ltwo}, \quad \gamma\in \left\{\frac{y}{4},\frac{y+1}{4}\right\}.
\end{aligned}
\]
We test \eqref{GalProj} with 
\[
v^n=(-\Deltaromean)^{\gamma-\frac12}\gcoefprojection \in \Wn \quad \text{for} \ \, \gamma\in \left\{\frac{y}{4},\frac{y+1}{4}\right\}
\]
 to obtain 
\[
\|(-\Deltaromean)^\gamma\gcoefprojection\|_{\Ltwo}\leq \|(-\Deltaromean)^\gamma\gcoef\|_{\Ltwo}.
\]
Then we make use of Lemma~\ref{lem:eigvaltransf} with $V=H=\Ltwo$, $C=(-\Deltaromean)^{-1}$, $\tilde{C}=(-\Delta_N)^{-1}$, $\mathcal{M}=(-\Delta_N)^{1/2}(-\Deltaromean)^{-1/2}$, and
\[
	\begin{aligned}
\|\mathcal{M}\|
=&\,\sup_{v\in \Ltwo\setminus\{0\}}\frac{\|(-\Delta_N)^{1/2}(-\Deltaromean)^{-1/2}v\|_{\Ltwo}}{\|v\|_{\Ltwo}}\\
=&\,\sup_{w\in H^1_\diamondsuit(\Omega)\setminus\{0\}}\frac{\|(-\Delta_N)^{1/2}w\|_{\Ltwo}}{\|(-\Deltaromean)^{1/2}w\|_{\Ltwo}}\\
=&\, \sup_{w\in H^1_\diamondsuit(\Omega)\setminus\{0\}}\frac{\|\nabla w\|_{\Ltwo}}{\|\sqrt{\romeaninv}\nabla w\|_{\Ltwo}}
\leq \|\romean\|_{\Linf},
	\end{aligned}
\]
where $H^1_\diamondsuit(\Omega)$ denotes the space of zero mean functions in $H^1(\Omega)$. Likewise 
$\|\mathcal{M}^{-1}\|\leq \|\romeaninv\|_{\Linf}$. Using \eqref{def Agamma}, we thus obtain
\[
\begin{aligned}
\|(-\Deltaromean)^\gamma\gcoefprojection\|_{\Ltwo}
	\leq \|(-\Deltaromean)^\gamma\gcoef\|_{\Ltwo}.
\end{aligned}
\]
Combining the bounds leads to
\[
	\begin{aligned}
\|(-\Delta_N)^\gamma\gcoefprojection\|_{\Ltwo}
&\leq \|\romean\|_{\Linf}^\gamma\|(-\Deltaromean)^\gamma\gcoefprojection\|_{\Ltwo}\\
&\leq \|\romean\|_{\Linf}^\gamma\|(-\Deltaromean)^\gamma\gcoef\|_{\Ltwo}
\leq \|\romean\|_{\Linf}^\gamma \left\|\romeaninv\right\|_{\Linf}^\gamma\|(-\Delta_N)^\gamma\gcoef\|_{\Ltwo}.
	\end{aligned}
\]
Finally, we apply the Kato--Ponce-type estimate to further infer
\[
\|(-\Delta_N)^\gamma\gcoef\|_{\Ltwo}=\|(-\Delta_N)^\gamma[\nabla\ln\romean\cdot\bfun]\|_{\Ltwo}\lesssim 
\|\nabla\ln\romean\|_{\Linf} \|\bfun\|_{H^{2\gamma}(\Omega)}
+\|\nabla\ln\romean\|_{H^{2\gamma}(\Omega)} \|\bfun\|_{\Linf},
\]
from which then the second and third estimates in \eqref{stability bounds grpojection} follow.
\end{proof}
\vspace*{-5mm}
\nocite{*}

\bibliography{lit}

\end{document}